\documentclass{svmult}

\usepackage{amscd, amssymb}

\newcommand\bA{{\mathbb A}}

\newcommand\bF{{\mathbb F}}
\newcommand\bG{{\mathbb G}}

\newcommand\bP{{\mathbb P}}

\newcommand\bZ{{\mathbb Z}}

\newcommand\cG{{\mathcal G}}

\newcommand\cM{{\mathcal M}}

\newcommand\cO{{\mathcal O}}

\newcommand\fm{\mathfrak{m}}

\newcommand\aff{{\rm aff}}

\newcommand\ant{{\rm ant}}
\newcommand\charc{{\rm char}}

\newcommand\diag{{\rm diag}}

\newcommand\id{{\rm id}}
\newcommand\kernel{{\rm ker}}

\newcommand\op{{\rm op}}

\newcommand\rat{{\rm rat}}
\newcommand\red{{\rm red}}

\newcommand\Aut{{\rm Aut}}
\newcommand\CL{{\rm CL}}

\newcommand\End{{\rm End}}

\newcommand\Hilb{{\rm Hilb}}
\newcommand\Hom{{\rm Hom}}
\newcommand\Ima{{\rm Im}}
\newcommand\Kern{{\rm Ker}}
\newcommand\Lie{{\rm Lie}}

\newcommand\SL{{\rm SL}}
\newcommand\Spec{{\rm Spec}}

\begin{document}

\title*{On Algebraic Semigroups and Monoids}

\author{Michel Brion}
\institute{Michel Brion 
\at Institut Fourier, Universit\'e de Grenoble, France\\
\email{Michel.Brion@ujf-grenoble.fr}}

\maketitle
 
\setcounter{minitocdepth}{2}

\dominitoc

\abstract{We present some fundamental results on (possibly nonlinear)
algebraic semigroups and monoids. These include a version of 
Chevalley's structure theorem for irreducible algebraic monoids, 
and the description of all algebraic semigroup structures on curves 
and complete varieties.
}

\section{Introduction}
\label{sec:intro}

Algebraic semigroups are defined in very simple terms: they are
algebraic varieties endowed with a 
composition law which is associative and a morphism of varieties. 
So far, their study has focused on the class of linear algebraic 
semigroups, that is, of closed subvarieties of the space of 
$n \times n$ matrices that are stable under matrix multiplication; 
note that for an algebraic semigroup, being linear is equivalent
to being affine. The theory has been especially developed by 
Putcha and Renner for linear algebraic monoids, i.e., those having 
a neutral element (see the books \cite{Putcha88, Renner05}).

\smallskip

In addition, there has been recent progress on the structure of 
(possibly nonlinear) algebraic monoids: by work of Rittatore, 
the invertible elements of any irreducible algebraic 
monoid $M$ form an algebraic group $G(M)$, open in $M$ 
(see \cite[Thm.~1]{Rittatore98}). Moreover, $M$ is linear 
if and only if so is $G(M)$ (see \cite[Thm.~5]{Rittatore07}). 
Also, the structure of normal irreducible algebraic monoids 
reduces to the linear case, as shown by Rittatore and 
the author: any such monoid is a homogeneous fiber bundle 
over an abelian variety, with fiber a normal irreducible linear 
algebraic monoid (see \cite[Thm.~4.1]{Brion-Rittatore}, and 
\cite{Renner-Rittatore} for further developments). 
This was extended by the author to all irreducible monoids 
in characteristic $0$ (see \cite[Thm.~3.2.1]{Brion08}).

\smallskip

In this article, we obtain some fundamental results on 
algebraic semigroups and monoids, that include the above 
structure theorems in slightly more general versions. 
We also describe all algebraic semigroup structures 
on abelian varieties, irreducible curves and complete 
irreducible varieties. The latter result is motivated by 
a remarkable theorem of Mumford and Ramanujam: 
if a complete irreducible variety $X$ has a (possibly 
nonassociative) composition law $\mu$ with a neutral element, 
then $X$ is an abelian variety with group law $\mu$ (see 
\cite[Chap.~II, \S 4, Appendix]{Mumford}).

\smallskip

As in \cite{Putcha88, Renner05}, we work over an 
algebraically closed field of arbitrary characteristic
(most of our results extend to a perfect field without much
difficulty; this is carried out in Subsection \ref{subsec:asmopf}).
But we have to resort to somewhat more advanced methods 
of algebraic geometry, as the varieties under consideration 
are not necessarily affine. For example, to show that
every algebraic semigroup has an idempotent, we use an
argument of reduction to a finite field, while the 
corresponding statement for affine algebraic semigroups 
follows from linear algebra. Also, we occasionally use
some semigroup and monoid schemes (these are briefly 
considered in \cite[Chap.~II]{Demazure-Gabriel}), but we 
did not endeavour to study them systematically.

\smallskip

This text is organized as follows. 
Section \ref{sec:asm} presents general results on 
idempotents of algebraic semigroups and on invertible 
elements of algebraic monoids. Both topics are 
fairly interwowen: for example, the fact that every 
algebraic monoid having no nontrivial idempotent is a 
group (whose proof is again more involved than in the
linear case) implies a version of the Rees structure
theorem for simple algebraic semigroups. 
In Section \ref{sec:iam},
we show that the Albanese morphism of an irreducible 
algebraic monoid $M$ is a homogeneous fibration with fiber 
an affine monoid scheme. This generalization of 
the main result of \cite{Brion-Rittatore} is obtained 
via a new approach, based on the consideration of the 
universal homomorphism from $M$ to an algebraic 
group. In Section \ref{sec:asscv}, we describe 
all semigroup structures on certain classes of 
varieties. We begin with the easy case of abelian 
varieties; as an unexpected consequence, we show that 
all the maximal submonoids of a given irreducible 
algebraic semigroup have the same Albanese variety. 
Then we show that every irreducible semigroup of 
dimension $1$ is either an algebraic group or 
an affine monomial curve; this generalizes a result
of Putcha in the affine case (see 
\cite[Thm.~2.13]{Putcha80} and \cite[Thm.~2.9]{Putcha81}). 
We also describe all complete irreducible semigroups,
via another variant of the Rees structure theorem.
Next, we obtain two general rigidity results; one of them implies 
(in loose words) that the automorphisms of a complete variety 
are open and closed in the endomorphisms. This has applications 
to complete algebraic semigroups, and yields another approach 
to the above theorem of Mumford and Ramanujam. Finally, 
we determine all families of semigroup laws on a given 
complete irreducible variety.

\smallskip

This article makes only the first steps in the study
of (possibly nonlinear) algebraic semigroups and monoids,
which presents many open questions. From the viewpoint of
algebraic geometry, it is an attractive problem to describe
all algebraic semigroup structures on a given variety.
Our classes of examples suggest that the associativity 
condition imposes strong restrictions which might make
this problem tractable: for instance, the composition 
laws on the affine line are of course all the
polynomial functions in two variables, but those that
are associative are obtained from the maps
$(x,y) \mapsto 0$, $x$, $y$, $x + y$ or $xy$ by a change
of coordinate. From the viewpoint
of semigroup theory, it is natural to investigate the
structure of an algebraic semigroup in terms of its
idempotents and the associated (algebraic) subgroups. 
Here a recent result of Renner and the author
(see \cite{Brion-Renner}) asserting that every algebraic 
semigroup $S$ is strongly $\pi$-regular (i.e., for any 
$x \in S$, some power $x^m$ belongs to a subgroup) opens 
the door to further developments.

\bigskip 

\noindent
{\bf Notation and conventions.} 
Throughout this article, we fix an algebraically closed field $k$. 
A \emph{variety} is a reduced, separated scheme of finite type 
over $k$; in particular, varieties need not be irreducible. 
By a \emph{point} of a variety $X$, we mean a closed 
(or equivalently, $k$-rational) point; we may identify $X$ to 
its set of points equipped with the Zariski topology and with 
the structure sheaf $\cO_X$. \emph{Morphisms} of varieties are
understood to be $k$-morphisms.

The textbook \cite{Hartshorne} will be our standard reference for 
algebraic geometry, and \cite{Eisenbud} for commutative algebra.
We will also use the books \cite{Springer} and \cite{Mumford}
for some basic results on linear algebraic groups, resp. 
abelian varieties.

\section{Algebraic semigroups and monoids}
\label{sec:asm}

\subsection{Basic definitions and examples}
\label{subsec:asg}

\begin{definition}\label{def:mon}
An (abstract) \emph{semigroup} is a set $S$ equipped with an
associative composition law $\mu : S \times S \to S$.
When $S$ is a variety and $\mu$ is a morphism, we say that 
$(S,\mu)$ is an \emph{algebraic semigroup}.

A \emph{neutral (resp. zero) element} of a semigroup 
$(S,\mu)$ is an element $x_o \in S$ such that 
$\mu(x, x_o) = \mu(x_o, x) = x$ for all $x \in S$
(resp. $\mu(x,x_o) = \mu(x_o,x) = x_o$ for all $x \in S$).

An abstract (resp. algebraic) semigroup $(S,\mu)$ equipped with 
a neutral element $x_o$ is called an abstract (resp. algebraic) 
\emph{monoid}.

An \emph{algebraic group} is a group $G$ equipped with the structure 
of a variety, such that the group law $\mu$ and the inverse map
$\iota : G \to G$, $g \mapsto g^{-1}$ are morphisms. 
\end{definition}

Clearly, a neutral element $x_o$ of a semigroup $S$ is unique 
if it exists; we then denote $x_o$ by $1_S$, or just by $1$ 
if this yields no confusion. Likewise, a zero element is unique 
if it exists, and we then denote it by $0_S$ or $0$. Also, 
we simply denote the semigroup law $\mu$ by $(x,y) \mapsto xy$. 

\begin{definition}
A \emph{left ideal} of a semigroup $(S,\mu)$ is a subset $I$ of $S$ 
such that $xy \in I$ for any $x \in S$ and $y \in I$. 
\emph{Right ideals} are defined similarly; a \emph{two-sided ideal}
is of course a left and right ideal.
\end{definition}

\begin{definition} \label{def:hom}
Given two semigroups $S$ and $S'$, 
a \emph{homomorphism of semigroups} is a map $\varphi : S \to S'$ 
such that $\varphi(xy) = \varphi(x) \varphi(y)$ for all $x,y \in S$. 
When $S$ and $S'$ are monoids, we say that $\varphi$ is a 
\emph{homomorphism of monoids} if in addition $\varphi(1_S) = 1_{S'}$.

A \emph{homomorphism of algebraic semigroups} is a homomorphism
of semigroups which is also a morphism of varieties. Homomorphisms 
of algebraic monoids, resp. of algebraic groups, are defined similarly.
\end{definition}

\begin{definition}\label{def:idem}
An \emph{idempotent} of a semigroup $S$ is an element $e \in S$
such that $e^2 = e$. We denote by $E(S)$ the set of idempotents.
\end{definition}

Idempotents yield much insight in the structure of semigroups; 
this is illustrated by the following:

\begin{remark}\label{rem:idem}
(i) Let $\varphi : S \to S'$ be a homomorphism of semigroups. 
Then $\varphi$ sends $E(S)$ to $E(S')$; moreover, the fiber 
of $\varphi$ at an arbitrary point $x' \in S'$ is a subsemigroup 
of $S$ if and only if $x' \in E(S')$. 

(ii) Let $S$ be a semigroup, and $M \subseteq S$ a submonoid with
neutral element $e$. Then $M$ is contained in the subset
\[ \{ x \in S ~\vert~ e x = x e = x \} = 
\{ exe ~\vert~ x \in S \} =: eSe, \]
which is the largest submonoid of $S$ with neutral element $e$.
This defines a bijective correspondence between idempotents and
maximal submonoids of~$S$.

(iii) Let $S$ be a semigroup, and $e \in E(S)$. Then the subset
\[ Se  := \{ xe ~\vert~ x \in S \} = \{ x \in S ~\vert~ xe = x \} \]
is a left ideal of $S$, and the map
\[ \varphi: S \longrightarrow Se, \quad x \longmapsto xe \]
is a \emph{retraction} (i.e., $\varphi(x) = x$ for all $x \in Se$). 
The fiber of $\varphi$ at $e$,
\[ S_e := \{ x \in S ~\vert~ xe = e \}, \]
is a subsemigroup of $S$. Moreover, the restriction
\[ \psi := \varphi \vert_{eS} : eS \longrightarrow eS \cap Se = eSe \] 
is a \emph{retraction of semigroups}, that is,
$\psi(x) = x$ for all $x \in eSe$ and $\psi$ is a homomorphism
(indeed, $x e y e = xy e$ for all $x \in S$ and $y \in eS$).
The fiber of $\psi$ at $e$,
\[ eS_e := \{ x \in S ~\vert~ ex = x ~{\rm and}~ xe = e \}, \]
is a subsemigroup of $S$ with law $(x,y) \mapsto y$
(since $x y = x e y = ey = y$ for all $x,y \in e S_e$).

When $S$ is an algebraic semigroup, $E(S)$ is a closed subvariety.
Moreover, $Se$, $S_e$, $eSe$ and $eS_e$ are closed in $S$ as well, 
and $\varphi$ (resp. $\psi$) is a retraction of varieties
(resp. of algebraic semigroups). In particular, every maximal
abstract submonoid of $S$ is closed. 

Similar assertions hold for the right ideal $eS$ 
and the subsemigroups
\[ {_e{}S} := \{ x \in S ~\vert~ ex = e \}, \quad
{_e{}Se} := \{ x \in S ~\vert~ xe = x ~{\rm and}~ ex = e\}. \]
\end{remark}

An abstract semigroup may have no idempotent; for example, the set
of positive integers equipped with the addition. Yet we have:

\begin{proposition}\label{prop:idem}
Any algebraic semigroup has an idempotent.
\end{proposition}

\begin{proof}
\smartqed
We use a classical argument of reduction to a finite field.
Consider first the case where $k$ is the algebraic closure of 
a prime field $\mathbb{F}_p$. Then $S$ and $\mu$ are defined over some 
finite subfield $\bF_q$ of $k$, where $q$ is a power of $p$. Thus, 
for any $x \in S$, the powers $x^n$, where $n$ runs 
over the positive integers, form a finite subsemigroup of $S$. 
We claim that some $x^n$ is idempotent. Indeed, we have 
$x^a = x^b$ for some integers $a > b > 0$. Thus, 
$x^b = x^b x^{a - b} = x^{b + m(a-b)}$ 
for all $m > 0$. In particular,  $x^b = x^{b(a-b+1)}$. 
Multiplying by $x^{b(a - b - 1)}$, we obtain
$x^{b(a - b)} = x^{2b(a - b)}$, i.e., $x^{b(a-b)}$ is idempotent.

Next, consider the case of an arbitrary field $k$. Choose 
$x \in S$; then $S$, $\mu$ and $x$ are defined over some 
finitely generated subring $R$ of $k$. In the language of 
schemes, we have a scheme $S_R$ together with morphisms
\[ \varphi : S_R \to \Spec(R), \quad 
\mu_R : S_R \times_{\Spec(R)} S_R \to S_R \]
and with a section $x_R : \Spec(R) \to S_R$ of $\varphi$, 
such that the $k$-scheme $S_R \times_{\Spec(R)} \Spec(k)$ 
is isomorphic to $S$; moreover, this isomorphism identifies 
$\mu_R \times_{\Spec(R)} \Spec(k)$ to $\mu$, and 
$x_R \times_{\Spec(R)} \Spec(k) $ to $x$. Also, $S_R$ is 
a semigroup scheme over $\Spec(R)$ (that is, $\mu_R$ is 
associative in an obvious sense), and the morphism 
$\varphi$ is of finite type. Denote by $E(S_R)$ 
the subscheme of idempotents of $S_R$, i.e., $E(S_R)$ 
is the preimage of the diagonal in $S_R \times S_R$ 
under the morphism
$S_R \to S_R \times S_R$, $s \mapsto (\mu_R(s,s),s)$. 
Then $E(S_R)$ is a closed subscheme of $S_R$; let 
\[ \psi: E(S_R) \to \Spec(R) \]
be the restriction of $\varphi$. 

We claim that the image of $\psi$ contains all closed
points of $\Spec(R)$. Indeed, consider a maximal ideal $\fm$ of 
$R$; then $R/\fm$ is a finite field. By the first step, the semigroup
$S_R \times_{\Spec(R)} \Spec(R/\fm)$ (obtained from $S_R$ 
by reduction mod $\fm$) contains an idempotent; this yields the claim.

Since $R$ is noetherian and the morphism $\psi$ is of finite type, 
its image is constructible (see e.g. \cite[Exer.~II.3.19]{Hartshorne}). 
In view of the claim, it follows that this image contains the generic 
point of $\Spec(R)$ (see e.g. [loc.~cit., Exer.~II.3.18]), i.e., $E(S_R)$ 
has (not necessarily closed) points over the fraction field of $R$. 
Hence $E(S)$ has (closed) points over the larger algebraically 
closed field $k$.
\qed
\end{proof}
 
Combining the above proposition with Remark \ref{rem:idem} 
(i), we obtain:

\begin{corollary}\label{cor:idem}
Let $f : S \to S'$ be a surjective homomorphism of algebraic 
semigroups. Then $f(E(S)) = E(S')$.
\end{corollary}

We now present several classes of (algebraic) semigroups:

\begin{example}\label{ex:semi}
(i) Any set $X$ has two semigroup laws $\mu_l$, $\mu_r$ given by 
$\mu_l(x,y) := x$ (resp. $\mu_r(x,y) := y$) for all $x,y \in X$. 
For both laws, every element is idempotent and $X$ has no 
proper two-sided ideal.

Also, every point $x \in X$ defines a semigroup law $\mu_x$ by 
$\mu_x(y,z) := x$ for all $y,z \in X$. 
Then $x$ is the zero element; it is the unique idempotent, 
and the unique proper two-sided ideal as well. 

The maps $\mu_l$, $\mu_r$, $\mu_x$ ($x \in X$) will be called
the \emph{trivial semigroup laws} on $X$. When $X$ is a variety,
these maps are algebraic semigroup laws. Note that every morphism
of varieties $f: X \to Y$ yields a homomorphism of algebraic semigroups
$(X,\mu_r) \to (Y,\mu_r)$, and likewise for $\mu_l$. Also,
$f$ yields a homomorphism $(X,\mu_x) \to (Y,\mu_y)$, where 
$y := f(x)$.

\smallskip

(ii) Let $X$ be a set, $Y \subseteq X$ a subset, $\rho: X \to Y$ 
a retraction, and $\nu$ a semigroup law on $Y$. Then the map 
\[ \mu : X \times X \longrightarrow X, \quad 
(x_1,x_2) \longmapsto \nu(\rho(x_1), \rho(x_2)) \]
is easily seen to be a semigroup law on $X$. Moreover,
$\rho$ is a retraction of semigroups, and $E(X) = E(Y)$. 
If in addition $X$ is a variety, $Y$ is a closed subvariety
and $\rho$, $\nu$ are morphisms, then $(X,\mu)$ is an 
algebraic semigroup.

When $Y$ consists of a single point $x$, we recover the 
semigroup law $\mu_x$ of the preceding example. 

\smallskip

(iii) Given two semigroups $(S,\mu)$ and $(S',\mu')$, 
we may define a composition law $\nu$ on the disjoint union 
$S \sqcup S'$ by
\[ \nu(x,y) = 
\cases{
\mu(x,y)  & if $x,y \in S$, \cr
y  & if $x \in S$ and $y \in S'$, \cr
x  & if $x \in S'$ and $y \in S$, \cr
\mu'(x,y)  & if $x,y \in S'$.\cr} 
\]
One readily checks that $(S \sqcup S',\nu)$ is a semigroup; 
moreover, $(S,\mu)$ is a subsemigroup and $(S',\mu')$ is 
a two-sided ideal. Also, note that $E(S \sqcup S') = E(S) \sqcup E(S')$.

When $S$ (resp. $S'$) has a zero element $0_S$ (resp. $0_{S'}$),
consider the set $S \cup_0 S'$ obtained from $S \sqcup S'$ by
identifying $0_S$ and $0_{S'}$. One checks that $S \cup_0 S'$
has a unique semigroup law $\nu_0$ such that the natural
map $S \sqcup S' \to S \cup_0 S'$ is a homomorphism; moreover,
the image of $0_S$ is the zero element. Here again, $S$ is a 
subsemigroup of $S \cup_0 S'$, and $S'$ is a two-sided ideal; 
we have $E(S \cup_0 S') = E(S) \cup_0 E(S')$.

If in addition $(S,\mu)$ and $(S',\mu')$ are algebraic semigroups,
then so are $(S \sqcup S',\nu)$ and $(S \cup_0 S', \nu_0)$.
This construction still makes sense when (say) $S'$ is a scheme 
of finite type over $k$, equipped with a closed point $0 = 0_{S'}$
and with the associated trivial semigroup law $\mu_0$. Taking for 
$S'$ the spectrum of a local ring of finite dimension as 
a $k$-vector space, and for $0$ the unique closed point of $S'$, 
we obtain many examples of nonreduced semigroup schemes 
(having a fat point at their zero element).

\smallskip

(iv) Any finite semigroup is algebraic. In the opposite direction, 
the (finite) set of connected components of an algebraic semigroup 
$(S,\mu)$ has a natural structure of semigroup. Indeed, if $C_1,C_2$
are connected components of $S$, then $\mu(C_1,C_2)$ is contained
in a unique connected component, say, $\nu(C_1,C_2)$. The resulting 
composition law $\nu$ on the set of connected components, $\pi_o(S)$, 
is clearly associative, and the canonical map $f : S \to \pi_o(S)$ is a 
homomorphism of algebraic semigroups. In fact, $f$ is the universal
homomorphism from $S$ to a finite semigroup.
\end{example}

Next, we present examples of algebraic monoids and of algebraic groups:

\begin{example}\label{ex:mono}
(i) Consider the set $M_n$ of $n \times n$ matrices with coefficients
in $k$, where $n$ is a positive integer. We may view $M_n$ as an affine 
space of dimension $n^2$; this is an irreducible algebraic monoid 
relative to matrix multiplication, the neutral element being of course 
the identity matrix. 

The subspaces $D_n$ of diagonal matrices, and $T_n$ of upper triangular
matrices, are closed irreducible submonoids of $M_n$. Note that 
$D_n$ is isomorphic to the affine $n$-space $\bA^n$ equipped with 
pointwise multiplication.

An example of a closed reducible submonoid of $M_n$ consists of 
those matrices having at most one nonzero entry in each row and each
column. This submonoid, that we denote by $R_n$, is the closure in 
$M_n$ of the group of monomial matrices (those having exactly one 
nonzero entry in each row and each column). Note that $R_n = D_n S_n$,
where $S_n$ denotes the symmetric group on $n$ letters, viewed as the
group of permutation matrices. Thus, the irreducible components 
of $R_n$ are parametrized by $S_n$. Each such component contains
the zero matrix; in particular, $R_n$ is connected.

\smallskip

(ii) A \emph{linear} algebraic monoid is a closed submonoid $M$ 
of some matrix monoid $M_n$. Then the variety $M$ is affine; 
conversely, every affine algebraic monoid is linear (see 
\cite[Thm.~II.2.3.3]{Demazure-Gabriel}). It follows that 
every affine algebraic semigroup is linear as well, see 
\cite[Cor.~3.16]{Putcha88}.

\smallskip

(iii) Let $A$ be a $k$-algebra (not necessarily associative, 
or commutative, or unital) and denote by $\End(A)$ the set 
of algebra endomorphisms of $A$. Then $\End(A)$, equipped with the
composition of endomorphisms, is an (abstract) monoid with zero.
Its idempotents are exactly the retractions of $A$ to subalgebras.
Given such a retraction $e : A \to B$, we have 
\[ e \End(A) \cong \Hom(A,B), \quad \End(A) e \cong \Hom(B,A),
\quad e \End(A) e \cong \End(B),\]
where $\Hom$ denotes the set of algebra homomorphisms.
Also, $\End(A)_e$ (resp. $_e{} \End(A)$) consists of those 
$\varphi \in \End(A)$ such that $\varphi(x) = x$ for all $x \in B$
(resp. $\varphi(x) - x \in I$ for all $x \in A$, where $I$ denotes 
the kernel of $e$).

If $A$ is finite-dimensional as a $k$-vector space, then 
$\End(A)$ is a linear algebraic monoid; indeed, it identifies 
to a closed submonoid of $M_n$, where $n := \dim(A)$.

\smallskip

(iv) Examples of algebraic groups include:

\noindent $\bullet$
the \emph{additive group} $\bG_a$, i.e., the affine line equipped with 
the addition, 

\noindent $\bullet$
the \emph{multiplicative group} $\bG_m$, i.e., the affine line minus 
the origin, equipped with the multiplication, 

\noindent $\bullet$
the \emph{elliptic curves}, i.e., the complete nonsingular irreducible 
curves of genus $1$, equipped with a base point; then there is a unique 
algebraic group structure for which this point is the neutral element, 
see e.g. \cite[Chap.~II, \S 4]{Hartshorne}). 

In fact, these examples yield all the connected algebraic groups
of dimension $1$, see \cite[Prop.~10.7.1]{Kempf}.

\smallskip

(v) A complete connected algebraic group is called an 
\emph{abelian variety}; elliptic curves are examples of such
algebraic groups. It is known that every abelian variety $A$ 
is a commutative group and a projective variety; moreover, 
the group law on $A$ is uniquely determined by the structure of
variety and the neutral element (see \cite[Chap.~II]{Mumford}).
\end{example}

\subsection{The unit group of an algebraic monoid}
\label{subsec:amo}

In this section, we obtain some fundamental results on the group 
of invertible elements of an algebraic monoid. We shall need the 
following observation:

\begin{proposition}\label{prop:neu}
Let $(M,\mu)$ be an algebraic monoid. Then $M$ has a unique irreducible 
component containing $1$: the {\rm neutral component} $M^o$. 
Moreover, $M^o X = X M^o = X$ for any irreducible component $X$ 
of $M$; in particular, $M^o$ is a closed submonoid of~$M$.
\end{proposition}

\begin{proof}
\smartqed
Let $X$, $Y$ be irreducible components of $M$. 
Then $XY$ is the image of the restriction of $\mu$ to $X \times Y$, 
and hence is a constructible subset of $M$; moreover, its closure 
$\overline{XY}$ is an irreducible subvariety of $M$. If $1 \in X$,
then $Y \subseteq XY \subseteq \overline{XY}$. Since $Y$ is an 
irreducible component, we must have $Y = XY = \overline{XY}$; 
likewise, one obtains that $YX = Y$. In particular, $X X = X$, 
i.e., $X$ is a closed submonoid. If in addition $1 \in Y$, 
then also $XY = YX = Y$, hence $Y = X$. This yields our assertions.
\qed
\end{proof}

\begin{remark}\label{rem:irr}
Any algebraic group $G$ is a nonsingular variety, and hence every
connected component of $G$ is irreducible. Moreover, the neutral
component $G^o$ is a closed normal subgroup, and the quotient group
$G/G^o$ parametrizes the components of $G$. 

In contrast, there exist connected reducible algebraic monoids:
for example, the monoid $R_n$ of Example \ref{ex:mono} (i).
Also, algebraic monoids are generally singular; for example,
the zero locus of $z^2 - xy$ in $\bA^3$ equipped with pointwise 
multiplication.

On a more advanced level, note that any group scheme is reduced
in characteristic $0$ (see e.g. 
\cite[Thm.~II.6.1.1]{Demazure-Gabriel}). In contrast, there always 
exist nonreduced monoid schemes. For example, one may stick 
an arbitrary fat point at the origin of the multiplicative monoid 
$(\bA^1,\times)$, by the construction of Example \ref{ex:semi} 
(iii). 
\end{remark}

\begin{definition}\label{def:unit}
Let $M$ be a monoid and let $x,y \in M$. Then $y$ is a \emph{left}
(resp. \emph{right}) \emph{inverse} of $x$ if $yx = 1$ 
(resp. $xy = 1$). We say that $x$ is \emph{invertible}
(also called a \emph{unit}) if it has a left and a right inverse.
\end{definition}

With the above notation, one readily checks that the left and right
inverses of any unit $x \in M$ are equal. 
Moreover, if $x' \in M$ is another unit with inverse $y'$, then $xy'$ 
is a unit with inverse $x'y$. Thus, the invertible elements of $M$ 
form a subgroup: the \emph{unit group}, that we denote by $G(M)$.

The following result on unit groups of algebraic monoids 
is due to Rittatore in the irreducible case (see 
\cite[Thm.~1]{Rittatore98}). The proof presented here follows
similar arguments.

\begin{theorem}\label{thm:unit}
Let $M$ be an algebraic monoid. Then $G(M)$ is an algebraic group,
open in $M$. In particular, $G(M)$ consists of nonsingular 
points of $M$.
\end{theorem}

\begin{proof}
\smartqed
Let
\[ G := \{ (x,y) \in M \times M^{\op} ~\vert~ xy = yx = 1 \}, \]
where $M^{\op}$ denotes the \emph{opposite} monoid to $M$, i.e.,
the variety $M$ equipped with the composition law 
$(x,y) \mapsto yx$. One readily checks that $G$ 
(viewed as a closed subvariety of $M \times M^{\op}$)
is a submonoid; moreover, every $(x,y) \in G$ has inverse $(y,x)$. 
Thus, $G$ is a closed algebraic subgroup of $M \times M^{\op}$.

The first projection $p_1 : M \times M^{\op} \to M$ restricts to a
homomorphism of monoids $\pi : G \to M$ with image being the unit
group $G(M)$. In fact, $G$ acts on $M$ by left multiplication: 
$(x,y) \cdot z := xz$, and $\pi$ is the orbit map 
$(x,y) \mapsto (x,y) \cdot 1$; in particular, $G(M)$ is the $G$-orbit 
of $1$. The isotropy subgroup of $1$ in $G$ is clearly trivial 
as a set. We claim that this also holds as a scheme; in other words, 
the isotropy Lie algebra of $1$ is trivial as well. 

To check this, recall that the Lie algebra $\Lie(G)$ is the
Zariski tangent space $T_{(1,1)}(G)$ and hence is contained in
$T_{(1,1)} (M \times M^{\op}) \cong T_1(M) \times T_1(M)$. Since
the differential at $(1,1)$ of the monoid law
$\mu : M \times M \to M$ is the map 
\[ T_1(M) \times T_1(M) \longrightarrow T_1(M), 
\quad (x,y) \longmapsto x + y, \]
we have 
\[ T_{(1,1)}(G) \subseteq 
\{ (x,y) \in T_1(M) \times T_1(M) ~\vert~ x + y = 0 \}. \]
Thus, the first projection $\Lie(G) \to T_1(M)$ is injective; but 
this projection is the differential of $\pi$ at $(1,1)$.
This proves our claim.

By that claim, $\pi$ is a locally closed immersion.
Thus, $G(M)$ is a locally closed subvariety of $M$, and $\pi$
induces an isomorphism of groups $G \cong G(M)$. So $G(M)$ is
an algebraic group.

It remains to show that $G(M)$ is open in $M$; it suffices to check
that $G(M)$ contains an open subset $U$ of $M$ (then the translates
$g U$, where $g \in G(M)$, form a covering of $G(M)$ by open
subsets of $M$). For this, we may replace $M$ with its neutral 
component $M^o$ (Proposition \ref{prop:neu}) and hence assume that
$M$ is irreducible. Note that 
\[ G(M) = \{ x \in M ~\vert~ 
x y = z x = 1 ~{\rm for}~{\rm some}~ y, z \in M \} \]
(then $y = z x y = z$). In other words,
\[ G(M) = p_1(\mu^{-1}(1)) \cap p_2(\mu^{-1}(1)), \]
where $p_1, p_2 : M \times M \to M$ denote the projections.
Also, the set-theoretic fiber at $1$ of the restriction 
$p_1 : \mu^{-1}(1) \to M$ consists of the single point $1$. 
By a classical result on the dimension of fibers 
of a morphism (see \cite[Exer.~II.3.22]{Hartshorne}),
it follows that every irreducible component $C$ 
of $\mu^{-1}(1)$ containing $1$ satisfies $\dim(C) = \dim(M)$,
and that the restriction $p_1: C \to M$ is dominant. Thus, 
$p_1(C)$ contains a dense open subset of $M$. Likewise, $p_2(C)$ 
contains a dense open subset of $M$, and hence so does $G(M)$.
\qed
\end{proof}

Note that the unit group of a linear algebraic monoid is linear,
see \cite[Cor.~3.26]{Putcha88}. Further properties of the unit 
group are gathered in the following:

\begin{proposition}\label{prop:unit}
Let $M$ be an algebraic monoid, and $G$ its unit group.

\begin{enumerate}
\item[{\rm (i)}]{If $x \in M$ has a left (resp. right) inverse, then 
$x \in G$.}
\item[{\rm (ii)}]{$M \setminus G$ is the largest proper two-sided ideal 
of $M$.}
\item[{\rm (iii)}]{If $1$ is the unique idempotent of $M$, then $M = G$.}
\end{enumerate}

\end{proposition}

\begin{proof}
\smartqed
(i) Assume that $x$ has a left inverse $y$. Then the left 
multiplication $M \to M$, $z \mapsto xz$ is an injective 
endomorphism of the variety $M$. By \cite[Thm.~C]{Ax} 
(see also \cite{Borel}), this endomorphism is surjective, 
and hence there exists $z \in M$ such that $xz = 1$. 
Then $y = yxz = z$, i.e., $x \in G$. The case where $x$ 
has a right inverse is handled similarly.

(ii) Clearly, any proper two-sided ideal of $M$ is contained 
in $M \setminus G$. We show that the latter is a two-sided 
ideal: let $x \in M \setminus G$ and $y \in M$. If $xy \in G$, 
then $y (xy)^{-1}$ is a right inverse of $x$. By (i), 
it follows that $x \in G$, a contradiction. Thus, 
$(M \setminus G) M \subseteq M \setminus G$. Likewise,
$M (M \setminus G) \subseteq M \setminus G$.

(iii) By Theorem \ref{thm:unit}, $M \setminus G$ is closed in $M$;
also, $M \setminus G$ is a subsemigroup of $M$ by (ii). Thus, if 
$M \neq G$ then $M \setminus G$ contains an idempotent, in view of
Proposition \ref{prop:idem}.
\qed
\end{proof}

\subsection{The kernel of an algebraic semigroup}
\label{subsec:kas}

In this subsection, we show that every algebraic semigroup 
has a smallest two-sided ideal (called its \emph{kernel}) 
and we describe the structure of that ideal, thereby generalizing 
some of the known results about the kernel of a linear algebraic 
semigroup (see \cite{Putcha88, Huang}). 

First, recall that the idempotents of any (abstract) semigroup 
$S$ are in bijective correspondence with the maximal submonoids
of $S$, via $e \mapsto eSe$, and hence with the maximal subgroups
of $S$, via $e \mapsto G(eSe)$. Thus, when $S$ is an algebraic
semigroup, its maximal subgroups are all locally closed in view
of Theorem \ref{thm:unit}. They are also pairwise disjoint: 
if $e \in E(S)$ and $x \in G(eSe)$, then there exists $y \in eSe$
such that $x y = y x = e$. Thus, $x S x \supseteq x y S y x = e S e$.
But $x S x \subseteq e S e S e S e \subseteq e S e$, and hence 
$x S x = e S e$. So $x S x$ is a closed submonoid of $S$ 
with neutral element $e$.

Next, we recall the classical definition of a partial 
order on the set of idempotents of any (abstract) semigroup:

\begin{definition}\label{def:order}
Let $S$ be a semigroup and let $e,f \in E(S)$. Then $e \leq f$
if we have $e = ef = fe$.
\end{definition}

Note that $e \leq f$ if and only if $e \in f S f$; this is also
equivalent to the condition that $e S e \subseteq f S f$. Thus, 
$\leq$ is indeed a partial order on $E(S)$ (this fact may 
of course be checked directly). Also, note that $\leq$ is 
preserved by every homomorphism of semigroups. For an 
algebraic semigroup, the partial order $\leq$ satisfies 
additional finiteness properties:

\begin{proposition}\label{prop:mini}
Let $S$ be an algebraic semigroup. 

\begin{enumerate}
\item[{\rm (i)}]{Every subset of $E(S)$ has a minimal element with 
respect to the partial order $\leq$, and also a maximal element.}
\item[{\rm (ii)}]{$e \in E(S)$ is minimal among all idempotents
if and only if $eSe$ is a group.}
\item[{\rm (iii)}]{If $S$ is commutative, then $E(S)$ is finite and 
has a smallest element.}
\end{enumerate}

\end{proposition}

\begin{proof}
\smartqed
(i) Note that the $e S e$, where $e \in E(S)$,
form a family of closed subsets of the noetherian topological
space $S$; hence any subfamily has a minimal element. For the 
existence of maximal elements, consider the family
\[ S \times_{Se} S := \{ (x,y) \in S \times S ~\vert~ x e = y e \} \]
of closed subsets of $S \times S$. Let $f \in E(S)$ such that
$e \leq f$. Then $S \times_{Sf} S \subseteq S \times_{Se} S$, 
since the equality $x f = y f$ implies that $x e = x f e = y f e = y e$. 
Moreover, if $S \times_{Se} S = S \times_{Sf} S$, then 
$(x,xe) \in S \times_{Sf} S$ for all $x \in S$, i.e., $xf = x e f$. 
Hence $x f = xe$, and $f = f^2 = fe = e$. Thus, a minimal 
$S \times_{Se} S$ (for $e$ in a given subset of $E(S)$) yields 
a maximal $e$.

(ii) Let $e$ be a minimal idempotent of $S$. 
Then $e$ is the unique idempotent of the algebraic monoid $eSe$. 
By Proposition \ref{prop:unit} (iii), it follows that $eSe$ is a group. 
The converse is immediate (and holds for any abstract semigroup).

(iii) Let $e,f$ be idempotents. Then $ef = fe$ is also idempotent,
and $ef = e(ef)e = f(ef)f$ so that $ef \leq e$ and $ef \leq f$.
Thus, any two minimal idempotents are equal, i.e., $E(S)$ has
a smallest element.

To show that $E(S)$ is finite, we may replace $S$ with its
closed subsemigroup $E(S)$, and hence assume that every element
of $S$ is idempotent. Then every connected component of $S$ is 
a closed subsemigroup in view of Example \ref{ex:semi} (iv).
So we may further assume that $S$ is connected. Let $x \in S$; 
then $x S$ is a connected commutative algebraic monoid 
with neutral element $x$, and consists of idempotents. 
Thus, $G(x S) = \{ x \}$. By Theorem \ref{thm:unit}, 
it follows that $x$ is an isolated point of $x S$. 
Hence $x S = \{ x \}$, i.e., $xy = x$ for all $y \in S$. 
Since $S$ is commutative, we must have $S = \{ x \}$.  
\qed
\end{proof}

As a consequence of the above proposition, every algebraic
semigroup admits minimal idempotents. These are of special
interest, as shown by the following:

\begin{proposition}\label{prop:SeS}
Let $S$ be an algebraic semigroup, $e \in S$ a minimal idempotent,
and $eSe$ the associated closed subgroup of $S$.

\begin{enumerate}
\item[{\rm (i)}]{The map
\[ \rho = \rho_e : S \longrightarrow S, \quad 
s \longmapsto s (ese)^{-1} s \]
is a retraction of varieties of $S$ to $SeS$. In particular,
$SeS$ is a closed two-sided ideal of $S$.}
\item[{\rm (ii)}]{The map
\[ \varphi: {_e{}Se} \times eSe \times eS_e \longrightarrow S,
\quad (x,g,y) \longmapsto xgy \]
yields an isomorphism of varieties to its image, $SeS$.}
\item[{\rm (iii)}]{Via the above isomorphism, the semigroup law on $SeS$ 
is identified to that on ${_e{}Se} \times eSe \times eS_e$ given by
\[ (x,g,y) (x',g',y') = (x,g \pi(y,x') g',y'), \]
where $\pi : eS_e \times {_e{}Se} \to eSe$ denotes the map
$(y,z) \mapsto yz$. This identifies the idempotents of $SeS$ 
to the triples $(x,\pi(y,x)^{-1},y)$, where 
$x \in {_e{}Se}$ and $y \in eS_e$. In particular, 
$E(SeS) \cong {_e{}Se} \times eS_e$ as a variety.}
\item[{\rm (iv)}]{The semigroup $SeS$ has no proper two-sided ideal.} 
\item[{\rm (v)}]{$SeS$ is the smallest two-sided ideal of $S$; 
in particular, it does not depend on the minimal idempotent $e$.} 
\item[{\rm (vi)}]{The minimal idempotents of $S$ are exactly the
idempotents of $SeS$.}
\end{enumerate}

\end{proposition}

\begin{proof}
\smartqed
(i) Clearly, $\rho$ is a morphism; also, since
$s (ese)^{-1} s \in S e S e S$ for all $s \in S$, the image of 
$\rho$ is contained in $SeS$. Let $z \in SeS$ and write 
$ z = set$, where $s,t \in S$; then
\[ \rho(z) = set (esete)^{-1} set 
= sete (ete)^{-1} (ese)^{-1} eset = set = z. \]
This yields the assertions.

(ii) Clearly, $\varphi$ takes values in $SeS$. Moreover, 
for any $s,t$ as above, we have 
\[ s e t = s e (ese)^{-1} esete (ete)^{-1} t = x g y , \]
where $x := se (ese)^{-1}$, $g := esete$ and $y := (ete)^{-1} et$.
Furthermore, $x \in {_e{}Se}$, $g \in eSe$ and $y \in eS_e$. 
In particular, the image of $\varphi$ is the whole $SeS$. 
Also, the map
\[ \psi : SeS \longrightarrow {_e{}Se} \times eSe \times eS_e, 
\quad s e t \longmapsto (x,g,y) \]
(where $x,g,y$ are defined as above) is a morphism of varieties
and satisfies $\varphi \circ \psi = \id$. Thus, it suffices to 
check that $\psi \circ \varphi = \id$. Let $x \in {_e{}Se}$, 
$g \in eSe$, $y \in eS_e$ and put $s := xgy$. Then $se = xg$
and $es = gy$. Hence $g = ese$, $x = se (ese)^{-1}$ and
$y = (ese)^{-1} es$, which yields the desired assertion.

(iii) For the first assertion, just write 
$(xgy)(x'g'y')= x (g(yx')g') y'$, and note that 
$yx' \in eS_e \, {_e{}Se} \subseteq eSe$. The assertions on 
idempotents follow readily.

(iv) Let $z \in SeS$ and write $z = \varphi(x,g,y)$. Then 
the subset $SzS$ of $SeS$ is identified with that of
${_e{}Se} \times eSe \times eS_e$ consisting of the triples
$(x_1, g_1 \pi(y_1,x) g \pi(y,x_2) g_2,y_2)$, where
$x_1,x_2 \in {_e{}Se}$, $g_1,g_2 \in eSe$ and $y_1,y_2 \in eS_e$.
It follows that $SzS = SeS$; in particular, $SzS$ contains $z$.
Hence $SeS$ is the smallest two-sided ideal containing $z$.

(v) Let $I$ be a two-sided ideal of $S$. Then $SeI$ is a
two-sided ideal of $S$ contained in $SeS$; hence $SeI = SeS$ 
by (iv). But $SeI \subseteq I$; this yields our assertions.

(vi) If $f \in E(S)$ is minimal, then $SfS = SeS$ by (v). 
Thus, $f \in SeS$. 

For the converse, let $f \in E(SeS)$. Then $SfS = SeS$ by (iv),
and hence $fSf = fSfSf = f(SeS)f$. Identifying $f$ to a triple
$(x,\pi(y,x)^{-1},y)$, one checks as in the proof of (iv) 
that $f(SeS)f$ is identified to the set of triples $(x,g,y)$, 
where $g \in eSe$. But $(x,\pi(y,x)^{-1},y)$ is the unique 
idempotent of this set. Thus, $f$ is the unique idempotent 
of $fSf$, i.e., $f$ is minimal.
\qed
\end{proof}

In view of these results, we may set up the following: 

\begin{definition}\label{def:ker}
The \emph{kernel} of an algebraic semigroup $S$ is the smallest
two-sided ideal of $S$, denoted by $\kernel(S)$.
\end{definition}

\begin{remark}\label{rem:simple}
(i) As a consequence of Proposition \ref{prop:SeS}, we see that any 
algebraic semigroup having no proper closed two-sided ideal is 
\emph{simple}, i.e., has no proper two-sided ideal at all. 
Moreover, any simple algebraic semigroup $S$, equipped with 
an idempotent $e$, is isomorphic (as a variety) to the product 
$X \times G \times Y$, where 
$X:= {_e{}Se}$ and $Y:=eS_e$ are varieties, and $G:= eSe$ is an 
algebraic group. This identifies $e$ to a point of the form 
$(x_o,1,y_o)$, where $x_o \in X$ and $y_o \in Y$; 
moreover, the semigroup law of $S$ is identified
to that as in Proposition \ref{prop:SeS} (iii), 
where $\pi: Y \times X \to G$ is a morphism such that 
$\pi(x_o,y) = \pi(x,y_o) = 1$ for all $x \in X$ and $y \in Y$. 

Conversely, any tuple $(X,Y,G,\pi,x_o,y_o)$ satisfying the above 
conditions defines an algebraic semigroup law on 
$S:= X \times G \times Y$ such that $e:= (x_o,1,y_o)$ is idempotent and 
${_e{}Se} = X \times \{ (1,y_o) \}$, 
$eSe = \{ x_o \} \times G \times \{ y_o \}$,
$eS_e = \{ (x_o,1) \} \times Y$.  

This description of algebraic semigroups having no proper 
closed two-sided ideal is a variant of the classical 
Rees structure theorem for those (abstract) semigroups 
that are \emph{completely simple}, that is, simple and having 
a minimal idempotent (see e.g. \cite[Thm.~1.9]{Putcha88}). 

(ii) By analogous arguments, one shows that every algebraic
semigroup $S$ contains minimal left ideals, and these are
exactly the subsets $Sf$, where $f$ is a minimal idempotent.
In particular, the minimal left ideals are all closed. Also,
given a minimal idempotent $e$, these ideals are exactly
the subsets $X \times G \times \{ y \}$ of $\kernel(S)$,
where $X := {_e{}Se}$, $G := eSe$ and $y \in Y:= eS_e$ as 
above. Similar assertions hold of course for the minimal
right ideals; it follows that the intersections of 
minimal left and minimal right ideals are exactly the subsets
$\{ x \} \times G \times \{ y \}$, where $x \in X$ and 
$y \in Y$. 
\end{remark}

\subsection{Unit dense algebraic monoids}
\label{subsec:udam}

In this subsection, we introduce and study the class of unit dense 
algebraic monoids. These include the irreducible algebraic monoids,
and will play an important role in their structure.

Let $M$ be an algebraic monoid, and $G(M)$ its unit group. Then 
the algebraic group $G(M) \times G(M)$ acts on $M$ via left and 
right multiplication: $(g,h) \cdot x := g x h^{-1}$. Moreover, 
the orbit of $1$ under this action is just $G(M)$, and 
the isotropy subgroup scheme of $1$ equals the diagonal,
$\diag(G(M)) := \{(g,g) ~\vert~ g \in G(M) \}$. 

\begin{definition}\label{def:ude}
An algebraic monoid $M$ is \emph{unit dense} if $G(M)$ is dense in $M$.
\end{definition}

For instance, every irreducible algebraic monoid is unit dense.
An example of a reducible unit dense algebraic monoid consists of 
$n \times n$ matrices having at most one nonzero entry in each 
row and each column (Example \ref{ex:mono} (i)).

Any unit dense monoid may be viewed as an equivariant embedding 
of its unit group, in the sense of the following:

\begin{definition}\label{def:equiv}
Let $G$ be an algebraic group. An \emph{equivariant embedding} of $G$
is a variety $X$ equipped with an action of $G \times G$ and with
a point $x \in X$ such that the orbit $(G \times G) \cdot x$ is dense
in $X$, and the isotropy subgroup scheme $(G \times G)_x$ is the 
diagonal, $\diag(G)$.
\end{definition}

Note that the law of a unit dense monoid is uniquely determined
by its structure of equivariant embedding, since that structure
yields the law of the unit group. Also, given an \emph{affine} 
algebraic group $G$, every \emph{affine} equivariant embedding 
of $G$ has a unique structure of algebraic monoid such that 
$G$ is the unit group, by \cite[Prop.~1]{Rittatore98}. Conversely,
every unit dense algebraic monoid with unit group $G$ is affine
by Theorem \ref{thm:aff} below. For an arbitrary connected 
algebraic group $G$, the equivariant embeddings of $G$ that admit 
a monoid structure are characterized in Theorem \ref{thm:monemb}
below.

\begin{proposition}\label{prop:neut}
Let $M$ be an algebraic monoid, and $G$ its unit group. Then the 
unit group of the neutral component $M^o$ is the neutral component 
$G^o$ of~$G$.

If $M$ is unit dense, then its irreducible components are exactly 
the subsets $gM^o$, where $g \in G$; they are indexed by $G/G^o$, 
the group of components of $G$.
\end{proposition}

\begin{proof}
\smartqed
Note that $G(M^o)$ is contained in $G$, and open in $M^o$ 
by Theorem \ref{thm:unit}. Hence $G(M^o)$ contains
an open neighborhood of $1$ in $M^o$, or equivalently in $G$. 
Using the group structure, it follows that $G(M^o)$ is open 
in $G$; also, $G(M^o)$ is irreducible since so is $M^o$. But 
the algebraic group $G$ contains a unique open irreducible 
subgroup: its neutral component. Thus, $G(M^o) = G^o$. 

Clearly, $gM^o$ is an irreducible component of $M$ for any $g \in G$,
and this component depends only on the coset $g G^o$. If $M$ is unit
dense, then any irreducible component $X$ of $M$ contains a unit,
say $g$. Since $g^{-1}X$ is an irreducible component containing $1$, 
it follows that $X = g M^o$. If $X = h M^o$ for some $h \in G$, then 
$g^{-1} h \in G \cap M^o$. Thus, $g^{-1}h G^o$ is an open subset of 
$M^o$, and hence meets $G^o$; so $g^{-1}h \in G^o$, i.e., $gG^o = hG^o$.
\qed
\end{proof}

\begin{proposition}\label{prop:sim}
Let $M$ be a unit dense algebraic monoid, and $G$ its unit group.
Then the kernel, $\kernel(M)$, is the unique closed orbit 
of $G \times G$ acting by left and right multiplication. 
Moreover, $\kernel(M) = G e G$ for any minimal idempotent $e$ of $M$.
\end{proposition}

\begin{proof}
\smartqed
We may choose a closed $G \times G$-orbit $Y$ in $M$. Then
\[ M Y M = \overline{G} Y \overline{G} \subseteq 
\overline{G Y G} = \overline{Y } = Y. \]
Thus, $Y$ is a two-sided ideal of $M$. Moreover, 
if $Z$ is another two-sided ideal, 
then $Z$ is stable by $G \times G$, and meets $Y$ since  
$YZ \subseteq Y \cap Z$. Thus, $Z$ contains $Y$; this shows that
$Y = \kernel(M)$. In particular, $Y$ is the unique closed 
$G \times G$-orbit; this proves the first assertion.
The second one follows from Proposition~\ref{prop:SeS}.
\qed
\end{proof}

\begin{proposition}\label{prop:clo}
Let $M$ be a unit dense algebraic monoid with unit group $G$.
Then the following conditions are equivalent for any $x \in M$:

\begin{enumerate}
\item[{\rm (i)}]{The orbit $Gx$ (for the $G$-action by left multiplication) 
is closed in $M$.}
\item[{\rm (ii)}]{$Gx = Mx$.}
\item[{\rm (iii)}]{$x \in \kernel(M)$.}
\end{enumerate}

\noindent
Moreover, all closed $G$-orbits in $M$ are equivariantly
isomorphic; in other words, the isotropy subgroup schemes $G_x$, 
where $x \in \kernel(M)$, are all conjugate. Also, each closed 
orbit contains a minimal idempotent. For any such idempotent 
$e$, the algebraic group $e M e$ equals $e G e$.
\end{proposition}

\begin{proof}
\smartqed
(i)$\Rightarrow$(ii) Since $Gx$ is closed in $M$, we have
$M x = \overline{G} x \subseteq \overline{G x} = Gx$
and hence $M x = Gx$.

(ii)$\Rightarrow$(iii) We have $G x = M x \supset \kernel(M)x$
and the latter subset is stable under left multiplication by $G$.
Hence $G x = \kernel(M)x$ is contained in $\kernel(M)$.

(iii)$\Rightarrow$(i) Let $e$ be a minimal idempotent of $M$. 
Since $\kernel(M) = G e G$, the $G$-orbits in $\kernel(M)$ 
are exactly the orbits $G eg$, 
where $g \in G$. Since the right multiplication by $g$ is 
an automorphism of the variety $M$ commuting with left 
multiplications, these orbits are all isomorphic as 
$G$-varieties. In particular, they all have the same dimension;
hence they are closed in $\kernel(M)$, and thus in $M$. 
Also, the orbit $G e g$ contains $g^{-1} e g$, 
which is a minimal idempotent
since the map $M \to M$, $x \mapsto g^{-1} x g$ is an automorphism
of algebraic monoids. Finally, we have $Ge = Me$ by (i);
likewise, $eG = eM$ and hence $eGe = eMe$.
\qed
\end{proof}

Note that the closed orbits for the left $G$-action are exactly
the minimal left ideals (considered in Remark \ref{rem:simple}
(ii) in the setting of algebraic semigroups).

\subsection{The normalization of an algebraic semigroup}
\label{subsec:norm}

In this subsection, we begin by recalling some background
results on the normalization of an arbitrary variety
(see e.g. \cite[\S 4.2, \S 11.2]{Eisenbud}). Then
we discuss the normalization of algebraic semigroups and
monoids; as in the previous subsection, this construction
will play an important role in their structure.

A variety $X$ is \emph{normal} at a point $x$ if the local ring 
$\cO_{X,x}$ is integrally closed in its total quotient ring;
$X$ is normal if it is so at any point. The normal points of 
a variety form a dense open subset, which contains the nonsingular 
points. The irreducible components of a normal variety are pairwise
disjoint, and each of them is normal.

An arbitrary variety $X$ has a \emph{normalization}, i.e., 
a normal variety $\tilde X$ together with a finite surjective 
morphism $\eta : \tilde X \to X$ which satisfies the following
universal property: for any normal variety $Y$ and any morphism 
$\varphi : Y \to X$ which is dominant (i.e., the image of $\varphi$
is dense in $X$), there exists a unique morphism 
$\tilde\varphi : Y \to \tilde X$ such that 
$\varphi = \eta \circ \tilde \varphi$. Then $\tilde X$ is uniquely
determined up to unique isomorphism, and $\eta$ is an isomorphism
above the open subset of normal points of $X$; in particular, 
$\eta$ is birational (i.e., an isomorphism over a dense open 
subset of $X$).

\begin{proposition}\label{prop:norm}
Let $(S,\mu)$ be an algebraic semigroup and let 
$\eta : \tilde S \to S$ be the normalization. 

\begin{enumerate}
\item[{\rm (i)}]{If the morphism $\mu : S \times S \to S$ is dominant, 
then $\tilde S$ has a unique algebraic semigroup law $\tilde \mu$ 
such that $\eta$ is a homomorphism. Moreover, 
$\eta(E(\tilde S)) = E(S)$.}
\item[{\rm (ii)}]{If $S$ is an algebraic monoid (so that $\mu$ is 
surjective), then $\tilde S$ is an algebraic monoid as well, 
with neutral element the unique preimage of $1_S$ under $\eta$. 
Moreover, $\eta$ induces an isomorphism $G(\tilde S) \cong G(S)$.}
\end{enumerate}

\end{proposition}

\begin{proof}
\smartqed
(i) By the assumption on $\mu$, the morphism
$\mu \circ (\eta \times \eta) : \tilde S \times \tilde S \to S$
is dominant. Since $\tilde S \times \tilde S$ is normal, there 
exists a unique morphism 
$\tilde \mu : \tilde S \times \tilde S \to \tilde S$ such that 
$\eta \circ \tilde \mu = \mu \circ (\eta \times \eta)$. 
Then $\tilde \mu$ is associative, since it coincides with $\mu$ 
on the dense open subset of normal points; moreover, $\eta$
is a homomorphism by construction. The assertion on idempotents
is a consequence of Corollary~\ref{cor:idem}.

(ii) The neutral element $1_S$ is a nonsingular point of $S$ 
by Theorem \ref{thm:unit}; thus, it has a unique preimage 
$1_{\tilde S}$ under $\eta$. Moreover, we have for any 
$\tilde x \in \tilde S$:
\[ \eta (\tilde \mu(\tilde x, 1_{\tilde S})) 
= \mu(\eta(\tilde x),\eta(1_{\tilde S}))
= \eta(\tilde x) = \eta(\tilde \mu(1_{\tilde S},\tilde x)). \]
Thus, 
$\tilde \mu(\tilde x,1_{\tilde S}) = 
\tilde \mu(1_{\tilde S}, \tilde x) = \tilde x$
for all $\tilde x$ such that $\eta(\tilde x)$ is a normal point
of $S$. By density of these points, it follows that $1_{\tilde S}$
is the neutral element of $(\tilde S,\tilde \mu)$. Finally, 
the assertion on unit groups follows from the inclusion 
$G(\tilde S) \subseteq \eta^{-1}(G(S))$ and from the fact that $\eta$ 
is an isomorphism above the nonsingular locus of $S$.
\qed
\end{proof}

\begin{remark}\label{rem:norm}
(i) For an arbitrary algebraic semigroup $S$, there may exist
several algebraic semigroup laws on the normalization $\tilde S$ 
that lift $\mu$. For example, let $x \in S$ and consider the 
trivial semigroup law $\mu_x$ of Example \ref{ex:semi} (i). Then 
$\mu_{\tilde x}$ lifts $\mu$ for any $\tilde x \in \tilde X$ such that 
$\eta(\tilde x) = x$. In general, such a point $\tilde x$ is not 
unique, e.g., when $S$ is a plane curve and $x$ an ordinary multiple 
point.

(ii) With the above notation, there may also exist no algebraic 
semigroup law on $\tilde S$ that lifts $\mu$. To construct examples 
of such algebraic semigroups, consider a normal irreducible affine 
variety $X$ and a complete nonsingular irreducible curve $C$, and
choose a finite surjective morphism $\varphi : C \to \bP^1$.
Let $Y := C \setminus \{ \varphi^{-1}(\infty) \}$; then $Y$ is an affine
nonsingular irreducible curve equipped with a finite surjective 
morphism $\varphi : Y \to \bA^1$. Choose a point $x_o \in X$ and let
$\gamma: Y \to X \times Y$, $y \mapsto (x_o,y)$; then $\gamma$
is a section of the second projection $p_2: X \times Y \to Y$.
By \cite[Thm.~5.1]{Ferrand}, there exists a unique irreducible
variety $S$ that sits in a co-cartesian diagram 
\[ \CD
Y @>{\varphi}>> \bA^1 \\
@V{\gamma}VV @V{\iota}VV \\
X \times Y @>{\eta}>> S. \\
\endCD \]
Then $\iota$ is a closed immersion, and $\eta$ is a finite morphism 
that restricts to an isomorphism 
$(X \setminus \{ x_o \}) \times Y \cong S \setminus \iota(\bA^1)$
and to the (given) morphism $\{ x_o \} \times Y \to \bA^1$, 
$(x_o,y) \mapsto \varphi(y)$.
In particular, $\eta$ is the normalization; $S$ is obtained by
``pinching $X \times Y$ along $\{ x_o \} \times Y$ via $\varphi$''.
Since the diagram
\[ \CD
Y @>{\varphi}>> \bA^1 \\
@V{\gamma}VV @V{\id}VV \\
X \times Y @>{\varphi \circ p_2}>> \bA^1 \\
\endCD \]
commutes, it yields a unique morphism $\rho : S \to \bA^1$ 
such that $\rho \circ \iota = \id$ and 
$\rho \circ \eta = \varphi \circ p_2$. The retraction $\rho$ 
defines in turn an algebraic semigroup law $\mu$ on $S$ by
$\mu(s,s') := \iota(\rho(s) \rho(s'))$ as in Example \ref{ex:semi} (ii). 

We claim that $\mu$ does not lift to any algebraic semigroup
law on $X \times Y$, if the curve $C$ is nonrational. Indeed, 
any such lift $\tilde \mu$ satisfies 
\[ \eta(\tilde \mu((x,y),(x',y'))) = \mu(\eta(x,y),\eta(x',y')) \]
\[ = \iota (\rho(\eta(x,y) \rho(\eta(x',y')))) 
= \iota(\varphi(y) \varphi(y')) \]
for any $x,x' \in X$ and any $y,y' \in Y$. 
As a consequence, $\tilde \mu((x,y),(x',y'))$ only depends on 
$(y,y')$, and this yields an algebraic semigroup law on $Y$
such that $\varphi$ is a homomorphism. But such a law does not
exist, as follows e.g. from Theorem \ref{thm:curve} below.
\end{remark}

\section{Irreducible algebraic monoids}
\label{sec:iam}

\subsection{Algebraic monoids with affine unit group}
\label{subsec:amaug}

The aim of this subsection is to prove the following result, 
due to Rittatore for irreducible algebraic monoids
(see \cite[Thm.~5]{Rittatore07}). The proof presented here
follows his argument closely, except for an intermediate step 
(Proposition \ref{prop:bun}).

\begin{theorem}\label{thm:aff}
Let $M$ be a unit dense algebraic monoid, and $G$ its unit group.
If $G$ is affine, then so is $M$.
\end{theorem}

\begin{proof}
\smartqed
Let $\eta : \tilde M \to M$ denote the normalization. Then $\tilde M$
is an algebraic monoid with unit group isomorphic to $G$,
by Proposition \ref{prop:norm}. Moreover, $G$ is dense in $\tilde M$
since it is so in $M$. If $\tilde M$ is affine, then $M$ is affine
by a result of Chevalley: the image of an affine variety by a finite
morphism is affine (see \cite[Exer.~II.4.2]{Hartshorne}). 
Thus, we may assume that $M$ is normal. Then $M$ is the disjoint 
union of its irreducible components, and each of them is isomorphic 
(as a variety) to the neutral component $M^o$ 
(Proposition \ref{prop:neut}). So we may assume in addition 
that $M$ is irreducible.

By Proposition \ref{prop:sim}, the connected algebraic group 
$G \times G$ acts on $M$ with a unique closed orbit. In view of 
a result of Sumihiro (see \cite{Sumihiro}), it follows that $M$ 
is quasiprojective; in other words, there exists a locally closed
immersion $\iota : M \to \bP^n$ for some positive integer $n$.
(We may further assume that $\iota$ is equivariant for some
action of $G$ on $\bP^n$; we will not need that fact in this proof).
Then the pull-back $L := \iota^*O_{\bP^n}(1)$ is an ample line bundle
on $M$. The associated principal $\bG_m$-bundle $\pi : X \to M$
(where $X$ is the complement of the zero section in $L$) is the
pull-back to $M$ of the standard principal $\bG_m$-bundle 
$\bA^{n+1} \setminus \{ 0 \} \to \bP^n$. Thus, $X$ is a locally closed
subvariety of $\bA^{n+1}$, and hence is quasi-affine.

By Proposition \ref{prop:bun} below (a version of 
\cite[Thm.~4]{Rittatore07}), $X$ has a structure of algebraic
monoid such that $\pi$ is a homomorphism. Since that monoid is
quasi-affine, it is in fact affine by a result of Renner 
(see \cite[Thm.~4.4]{Renner84}). Moreover, the map $\pi: X \to M$ is 
the categorical quotient by the action of $\bG_m$; hence $M$ is affine.
\qed
\end{proof}

\begin{proposition}\label{prop:bun}
Let $M$ be a normal irreducible algebraic monoid, and assume that
its unit group $G$ is affine. Let $\varphi : L \to M$ be a line bundle,
and $\pi : X \to M$ the associated principal $\bG_m$-bundle. Then
$X$ has a structure of a normal irreducible algebraic monoid
such that $\pi$ is a homomorphism. 
\end{proposition}

\begin{proof}
\smartqed
By \cite[Lem.~4.3]{KKV89}, the preimage $Y := \pi^{-1}(G)$ 
has a structure of algebraic group such that the restriction 
of $\pi$ is a homomorphism; we then have an exact sequence 
of algebraic groups
\[ 1 \longrightarrow \bG_m \longrightarrow Y 
\stackrel{\pi}{\longrightarrow} G \longrightarrow 1, \]
where $\bG_m$ is contained in the center of $Y$. Thus, the group 
law $\mu_G : G \times G \to G$ sits in a cartesian square
\[ \CD
Y \times^{\bG_m} Y @>{\mu_Y}>> Y \\
@V{\pi \times \pi}VV @V{\pi}VV \\
G \times G @>{\mu_G}>> G, \\
\endCD \]
where $Y \times^{\bG_m} Y$ denotes the quotient of $Y \times Y$ 
by the action of $\bG_m$ via $t \cdot (y,z) = (ty,t^{-1}z)$,
and $\mu_Y$ stands for the group law on $Y$.
Via the correspondence between principal $\bG_m$-bundles and
line bundles, this translates into a cartesian square 
\[ \CD
p_1^*(L\vert_G) \otimes p_2^*(L\vert_G) @>>> L\vert_G \\
@V{\varphi \times \varphi}VV @V{\varphi}VV \\
G \times G @>{\mu_G}>> G, \\
\endCD \]
where $p_1$, $p_2: G \times G \to G$ denote the projections. 
In other words, we have an isomorphism
\[ p_1^*(L\vert_G) \otimes p_2^*(L\vert_G) 
\stackrel{\cong}{\longrightarrow} \mu_G^*(L\vert_G) \]
of line bundles over $G \times G$. 

Over $M \times M$, this yields an isomorphism
\[ p_1^*(L) \otimes p_2^*(L) \stackrel{\cong}{\longrightarrow} 
\mu^*(L) \otimes \cO_{M \times M}(D), \]
where $D$ is a Cartier divisor with support in 
$(M \times M) \setminus (G \times G)$. Since $G$ is affine, 
the irreducible components $E_1, \ldots, E_n$ of $M \setminus G$ 
are divisors of $M$. Thus, the irreducible components of 
$(M \times M) \setminus (G \times G)$ are exactly the divisors
$E_i \times M$ and $M \times E_j$, where $i,j = 1,\ldots,n$. 
Hence 
\[ D = p_1^*(D_1) + p_2^*(D_2) \] 
for some Weil divisors $D_1$, $D_2$ with support in $M \setminus G$. 
In particular, the pullback of $D$ to $M \times G$ is 
$D_1 \times G$. Since $D$ is Cartier, so is $D_1$; likewise, $D_2$
is Cartier. We thus obtain an isomorphism
\[ p_1^*(L) \otimes p_2^*(L) \stackrel{\cong}{\longrightarrow} 
\mu^*(L) \otimes p_1^*(\cO_M(D_1)) \otimes p_2^*(\cO_M(D_2)) \]
of line bundles over $M \times M$.
We now pull back this isomorphism to $M \times \{ 1 \}$. Note that
$\mu^*(L)\vert_{M \times \{ 1 \}} = L= p_1^*(L)\vert_{M \times \{ 1 \}}$;
also, $p_1^*(\cO_M(D_1)) \vert_{M \times \{ 1 \}} = \cO_M(D_1)$, and both 
$p_2^*(L)\vert_{M \times \{ 1 \}}$, $p_2^*(\cO_M(D_2)) \vert_{M \times 1}$
are trivial. Thus, $\cO_M(D_1)$ is trivial; one shows similarly
that $\cO_M(D_2)$ is trivial. Hence we have in fact an isomorphism
\[ p_1^*(L) \otimes p_2^*(L) \stackrel{\cong}{\longrightarrow} 
\mu^*(L). \] 
As above, this translates into a cartesian square
\[ \CD
X \times^{\bG_m} X @>>> X \\
@V{\pi \times \pi}VV @V{\pi}VV \\
M \times M @>{\mu}>> M. \\
\endCD \]
In turn, this yields a morphism $\nu : X \times X \to X$
which lifts $\mu : M \times M \to M$ and extends the group law
$Y \times Y \to Y$. It follows readily that $\nu$ is associative
and has $1_Y$ as a neutral element.
\qed
\end{proof}

A noteworthy consequence of Theorem \ref{thm:aff} is 
the following sufficient condition for an algebraic monoid 
to be affine, which slightly generalizes 
\cite[Cor.~3.3]{Brion-Rittatore}:

\begin{corollary}\label{cor:zero}
Let $M$ be a unit dense algebraic monoid having a zero element.
Then $M$ is affine.
\end{corollary}

\begin{proof}
\smartqed
Consider the action of the unit group $G$ on $M$ via left 
multiplication. This action is faithful, and fixes the zero 
element. It follows that $G$ is affine (see e.g. 
\cite[Prop.~2.1.6]{BSU}). Hence $M$ is affine by Theorem 
\ref{thm:aff}.
\qed
\end{proof}

\subsection{Induction of algebraic monoids}
\label{subsec:ind}

In this subsection, we show that any unit dense algebraic monoid
has a universal homomorphism to an algebraic group, and we
study the fibers of this homomorphism.

\begin{proposition}\label{prop:uni}
Let $M$ be a unit dense algebraic monoid, and $G$ its unit group.

\begin{enumerate}
\item[{\rm (i)}]{There exists a homomorphism of algebraic monoids 
$\varphi : M \to \cG(M)$, where $\cG(M)$ is an algebraic group, 
such that every homomorphism of algebraic monoids 
$\psi : M \to \cG$, where $\cG$ is an algebraic group, 
factors uniquely as $\varphi$ followed by a homomorphism 
of algebraic groups $\cG(M) \to \cG$.}
\item[{\rm (ii)}]{We have $\cG(M) = \varphi(M) = \varphi(G) = G/H$, 
where $H$ denotes the smallest normal subgroup scheme 
of $G$ containing the isotropy subgroup scheme $G_x$ for some
$x \in \kernel(M)$.}
\end{enumerate}

\end{proposition}

\begin{proof}
\smartqed
We show both assertions simultaneously.
Let $\psi : M \to \cG$ be a homomorphism as in the statement. 
Then $\psi \vert_G$ is a homomorphism of algebraic groups, 
and hence its image is a closed subgroup of $\cG$. 
Since $M$ is unit dense, it follows that $\psi(M) = \psi(G)$. 
Let $K$ be the scheme-theoretic kernel of $\psi \vert_G$.
Then $K$ is a normal subgroup scheme of $G$, 
and $\psi$ induces an isomorphism from $G/K$ to $\psi(G)$; 
we may thus view $\psi$ as a $G$-equivariant homomorphism $M \to G/K$.  
In particular, for any $x \in M$, the map 
$g \mapsto \psi(g \cdot x)$ yields a morphism $G \to G/K$ which is 
equivariant under the action of $G$ by left multiplication, and
invariant under the action of the isotropy subgroup scheme $G_x$ by
right multiplication. Thus, $K$ contains $G_x$; hence $K$ contains $H$,
and $\psi \vert_G$ factors as the quotient homomorphism
$\gamma : G \to G/H$ followed by the canonical homomorphism
$\pi : G/H \to G/K$.

Next, choose $x \in \kernel(M)$; then $Gx = Mx$ by 
Proposition \ref{prop:clo}. Thus, the morphism $M \to M x$, 
$y \mapsto y x$ may be viewed as a morphism $M \to G x \cong G/G_x$. 
Composing with the morphism $G/G_x \to G/H$ induced by the inclusion 
of $G_x$ in $H$, we obtain a morphism $\varphi : M \to G/H$. Clearly, 
$\varphi$ is $G$-equivariant, and $\varphi(1)$ is the neutral element
of $G/H$. Thus, the restriction $\varphi \vert_G$ is the quotient 
homomorphism $\gamma$. By density, $\varphi$ is a homomorphism 
of monoids, and $\psi = \pi \circ \varphi$. So $\varphi$ is the 
desired homomorphism.
\qed
\end{proof}

\begin{remark}\label{rem:uni}
(i) As a consequence of the above proposition, the smallest 
subgroup scheme of $G$ containing $G_x$ is independent of the
choice of $x \in \kernel(M)$. This also follows from the fact 
that the subgroup schemes $G_x$, where $x \in \kernel(M)$, are all
conjugate in $G$ (Proposition \ref{prop:clo}). By that proposition,
we may take for $x$ any minimal idempotent of~$M$.

(ii) As another consequence, 
any irreducible semigroup $S$ has a universal homomorphism
to an algebraic group (in the sense of the above proposition). 
Indeed, choose an idempotent $e$ in $S$, and
consider a homomorphism of semigroups $\psi : S \to \cG$, where
$\cG$ is an algebraic group. Then $\psi(x) = \psi(exe)$ for all 
$x \in S$; moreover, $eSe$ is an irreducible monoid with neutral
element $e$. Thus, there exists a unique homomorphism 
$\pi : \cG(eSe) \to \cG$ such that $\psi(x) = \pi (\phi(exe))$
for all $x \in S$, where $\phi : eSe \to \cG(eSe)$ denotes 
the universal homomorphism. Then we must have
$\pi(\phi(exye))= \pi(\phi(exe) \phi(eye))$ for all 
$x,y \in S$. Let $H$ denote the smallest normal subgroup 
scheme of $\cG(eSe)$ containing the image of the morphism
\[ S \times S \longrightarrow \cG(eSe), \quad
(x,y) \longmapsto \phi(exye) \phi(exe)^{-1} \phi(eye)^{-1}, \]
and let $\varphi : S \to \cG(eSe)/H$ denote the homomorphism
that sends every $x$ to the image of $exe$. Then
$\pi$ factors as $\varphi$ followed by a unique homomorphism
of algebraic groups $\cG(eSe)/H \to \cG$, i.e., $\varphi$
is the desired homomorphism. Note that 
$\varphi : S \to \cG(S)$ is surjective by construction; 
in particular, $\cG$ is connected.
\end{remark}

\begin{proposition}\label{prop:red}
Keep the notation and assumptions of Proposition \ref{prop:uni}.

\begin{enumerate}
\item[{\rm (i)}]{If $H$ is an algebraic group (e.g., if $\charc(k) = 0$), 
then the scheme-theoretic fibers of $\varphi$ are reduced.}
\item[{\rm (ii)}]{If $M$ is normal, then $H$ is a connected 
algebraic group; moreover, the scheme-theoretic fibers of 
$\varphi$ are reduced and irreducible.}
\end{enumerate}

\end{proposition}

\begin{proof}
\smartqed
(i) Denote by $\gamma : G \to G/H$ the quotient homomorphism and
form the cartesian square
\[ \CD
X @>{\varphi'}>> G \\
@V{\gamma'}VV @V{\gamma}VV \\
M @>{\varphi}>> G/H. \\
\endCD \]
Since $\gamma$ and $\varphi$ are equivariant for the actions of $G$
by left multiplication, $X$ is equipped with a $G$-action such that 
$\gamma'$ and $\varphi'$ are equivariant. Denote by $N$ the 
(scheme-theoretic) fiber of $\varphi'$ at the neutral element $1_G$. 
Then the morphism
\[ G \times N \longrightarrow X, \quad (g,x) \longmapsto g \cdot x \]
is an isomorphism with inverse given by 
$x \mapsto (\varphi'(x), \varphi'(x)^{-1} \cdot x)$. 
Moreover, the fiber of $\varphi'$ at every $g \in G$ is 
$g \cdot N \cong N$. If $H$ is an algebraic group (i.e., 
if $H$ is smooth; this holds when $\charc(k)=0$), then
the morphism $\gamma$ is smooth as well; hence so is $\gamma'$. 
It follows that $X$ is reduced. But $X \cong G \times N$ 
and hence $N$ is reduced. 
If in addition $H$ is connected, then the fibers of $\gamma$ are 
irreducible; hence the same holds for $\gamma'$, and $X$ is 
irreducible. As above, it follows that $N$ is irreducible. 

(ii) Consider the reduced neutral component $H^o_{\red} \subseteq H$; 
then $H^o_{\red}$ is a closed normal subgroup of $G$. 
Moreover, the natural map $\delta : G/H^o_{\red} \to G/H$ 
is a finite morphism and sits in a commutative square
\[ \CD
G @>{\iota}>> M \\
@V{\epsilon}VV @V{\varphi}VV \\
G/H^o_{\red} @>{\delta}>> G/H, \\
\endCD \]
where $\iota$ denotes the inclusion.
Let $\Gamma \subseteq M \times G/H^o_{\red}$ be the closure of the graph 
of $\epsilon$. Then the projection $p_1 : \Gamma \to M$ is a finite 
morphism, and an isomorphism over the dense open subset $G$ of $M$. 
Since $M$ is normal, it follows that $p_1$ is an isomorphism, 
i.e., $\epsilon$ extends to a morphism $\psi : M \to G/H^o_{\red}$.  
As $\epsilon$ is a homomorphism of algebraic groups, $\psi$ must
be a homomorphism of monoids. Thus, $\psi$ factors through $\varphi$, 
and hence $H^o_{\red} = H$. In other words, $H$ is a connected algebraic 
group. 
\qed
\end{proof}

But in general, the scheme-theoretic fibers of the homomorphism
$\varphi : M \to \cG(M)$ are reducible; also, these fibers are 
nonreduced in prime characteristics, as shown by the following:

\begin{example}\label{ex:nonred}
Consider the monoid $\bA^3$ equipped with pointwise multiplication,
and the locally closed subset
\[ M := \{ (x,y,z) ~\vert~ z^n = x y^n ~{\rm and}~ x \neq 0 \}, \]
where $n$ is a positive integer. Then $M$ is an irreducible 
commutative algebraic monoid with unit group 
\[ G = \{ (x,y,z) ~\vert~ z^n = x y^n ~{\rm and}~ z \neq 0 \}, \]
isomorphic to $\bG_m^2$ via the projection $(x,y,z) \mapsto (y,z)$.
Moreover, $\kernel(M) = M e$, where $e := (1,0,0)$ is the unique
minimal idempotent. Since $M$ is commutative, the isotropy subgroup 
scheme $H$ of Proposition \ref{prop:uni}
is just $G_e$; the latter is the scheme-theoretic kernel 
of the homomorphism $x : G \to \bG_m$. Thus,
\[ H \cong \{ (y,z) \in \bG_m^2 ~\vert~ y^n = z^n \} 
\cong \bG_m \times \mu_n, \]
where $\mu_n$ denotes the subgroup scheme of $n$th roots
of unity. The universal homomorphism $\varphi : M \to G/H$ 
is identified to $x: M \to \bG_m$, and this identifies the fiber
of $\varphi$ at $1$ to the submonoid scheme $(y^n = z^n)$ of 
$(\bA^2, \times)$. The latter scheme is reducible when $n \geq 2$, 
and nonreduced when $n$ is a multiple of $\charc(k)$.
\end{example}
 
We keep the notation and assumptions of Proposition \ref{prop:uni},
and denote by $N$ the scheme-theoretic fiber of $\varphi$ at $1$.
Assume in addition that \emph{$H$ is an algebraic group}
(this holds e.g. if $M$ is normal or if $\charc(k)=0$). 
Then $N$ is reduced by Proposition \ref{prop:red};
also, $N$ is a closed submonoid of $M$, containing $H$ and stable 
under the action of $G$ on $M$ by conjugation (via 
$g \cdot x := g x g^{-1}$). Moreover, the map 
\[ \pi : G \times N \longrightarrow M, \quad (g,y) \longmapsto g y \]
is a homomorphism of algebraic monoids, where $G \times N$ is equipped 
with the composition law
\[ (g_1,y_1) \, (g_2,y_2) := (g_1g_2, g_2^{-1}y_1 g_2 y_2) \]
with unit $(1_G,1_N)$ (this defines the \emph{semi-direct product
of $G$ with $N$}). Finally, $\pi$ is the quotient morphism for 
the action of $H$ on $G \times N$ via 
\[ h \cdot (g,y) := (gh^{-1}, hy). \]
In other words, the morphism $\varphi : M \to G/H$ identifies $M$
to the fiber bundle $G \times^H N \to G/H$
associated to the principal $H$-bundle $G \to G/H$ and to the variety 
$N$ on which $H$ acts by left multiplication. We say that the algebraic 
monoid $M$ is \emph{induced} from $N$.

If we no longer assume that $H$ is an algebraic group, then 
$N$ is just a submonoid scheme of $M$, and
the above properties hold in the setting of monoid schemes. 
We now obtain slightly weaker versions of these properties 
in the setting of algebraic monoids.  

\begin{proposition}\label{prop:nred}
Let $M$ be a unit dense algebraic monoid, $G$ its unit group,
$\varphi : M \to G/H$ the universal homomorphism to an algebraic
group, and $N$ the scheme-theoretic fiber of $\varphi$ at $1$.
Denote by $H_{\red}$ (resp. $N_{\red}$) the reduced scheme of $H$ 
(resp. $N$).

\begin{enumerate}
\item[{\rm (i)}]{$H_{\red}$ is a closed normal subgroup of $G$ and 
$N_{\red}$ is a closed submonoid of $M$, stable under
the action of $G$ by conjugation.}
\item[{\rm (ii)}]{$G \times^{H_{\red}} N_{\red}$ is an algebraic monoid, 
and the natural map 
\[ \psi : G \times^{H_{\red}} N_{\red} \longrightarrow M \]
is a finite bijective homomorphism of algebraic monoids.}
\item[{\rm (iii)}]{$N_{\red}$ is unit dense and its unit group is $H_{\red}$.}
\item[{\rm (iv)}]{$\psi$ is birational.}
\end{enumerate}

\end{proposition}

\begin{proof}
\smartqed
(i) The assertion on $H_{\red}$ is well-known. That on $N_{\red}$ 
follows readily from the fact that $N$ is a closed submonoid 
scheme of $M$, stable under the $G$-action by conjugation.

(ii) The natural map $G/H_{\red} \to G/H$ is a purely inseparable
homomorphism of algebraic groups, and hence is finite and bijective.
Also, $G\times^{H_{\red}} N$ is the fibered product of
$M = G\times ^H N$ and $G/H_{\red}$ over $G/H$. Thus,
$G \times^{H_{\red}} N$ is a monoid scheme; moreover, the natural 
morphism $G \times^{H_{\red}} N \to M$ is finite, and bijective on
closed points. As 
$G \times^{H_{\red}} N_{\red} = (G \times^{H_{\red}} N)_{\red}$,
this yields our assertions. 

(iii) Since $M$ is unit dense with unit group $G$ and $\psi$ 
is a homeomorphism, we see that $G \times^{H_{\red}} N_{\red}$ 
is unit dense with unit group $G$ as well. It follows that 
$G \times N_{\red}$ is unit dense with unit group $G \times H_{\red}$.
Thus, $H_{\red}$ is the unit group of $N_{\red}$ and is dense there.

(iv) Just note that $\psi$ restricts to the natural isomorphism 
$G \times^{H_{\red}} H_{\red} \stackrel{\cong}{\to} G$; moreover, 
$G \times^{H_{\red}} H_{\red}$ is a dense open subset of
$G \times^{H_{\red}} N_{\red}$.
\qed
\end{proof}

\begin{example}\label{ex:nred}
Assume that $\charc(k) = p > 0$. Consider the monoid
\[ M := \{ (x,y,z) \in \bA^3 ~\vert~ 
z^p = x y^p ~{\rm and}~ x \neq  0 \} \]
relative to pointwise multiplication, as in Example 
\ref{ex:nonred}. Recall from this example that $G \cong \bG_m^2$
and that the universal homomorphism $\varphi : M \to G/H$
is just $x : M \to \bG_m$, with scheme-theoretic fiber $N$ at $1$ 
being the submonoid scheme $(z^p = y^p)$ of $(\bA^2, \times)$.
It follows that $N_{\red} \cong (\bA^1, \times)$, 
$H_{\red} \cong \bG_m$ and 
$G \times^{H_{\red}} N_{\red} \cong \bG_m \times \bA^1$,
where the right-hand side is equipped with pointwise multiplication.
One checks that $\psi : G \times^{H_{\red}} N_{\red} \to M$
is identified with the map $(t,u) \mapsto (t^p, t,u)$.
\end{example}

Returning to the general setting, we now relate the idempotents
and kernel of $M$ with those of $N_{\red}$:

\begin{proposition}\label{prop:small}
Keep the notation and assumptions of Proposition
\ref{prop:nred}.

\begin{enumerate}
\item[{\rm (i)}]{$E(M) = E(N_{\red})$.}
\item[{\rm (ii)}]{The assignment $I \mapsto I \cap N_{\red}$ defines 
a bijection between the two-sided ideals of $M$ and those 
two-sided ideals of $N_{\red}$ that are stable under conjugation 
by $G$. The inverse bijection is given by $J \mapsto GJ$.}
\item[{\rm (iii)}]{We have $\kernel(M) \cap N_{\red} = \kernel(N_{\red})$ 
and $\kernel(M) = G \, \kernel(N_{\red})$.}
\end{enumerate}

\end{proposition}

\begin{proof}
\smartqed
(i) Clearly, $E(N_{\red}) \subseteq E(M)$. Moreover, if $e \in E(M)$ 
then $\varphi(e) = 1_{G/H}$ and hence $e \in N$, i.e., $e \in N_{\red}$.

(ii) Consider a two-sided ideal $I$ of $M$. Then $J := I \cap N_{\red}$ 
is a two-sided ideal of $N_{\red}$, stable under conjugation
by $G$ (since so are $I$ and $N_{\red}$). Moreover, $I = GJ$, since
$M = GN_{\red}$ and $I = GI$. 

Conversely, let $J$ be a two-sided ideal of $N_{\red}$, 
stable under conjugation by $G$. Then $I := GJ$
is closed in $M$ and satisfies $I \cap N_{\red} = J$, as follows
easily from the fact that $\psi : G \times^{H_{\red}} N_{\red} \to M$ 
is a homeomorphism.
Moreover, $GIG = GJG = GJ = I$ by stability of $J$ under conjugation. 
Since $M$ is unit dense and $I$ is closed in $M$, it follows that 
$MIM = I$; in other words, $I$ is a two-sided ideal. 

(iii) Since $N_{\red}$ is stable under conjugation by $G$, so is 
$\kernel(N_{\red})$. In view of (ii), it follows that
$\kernel(M) \cap N_{\red} = \kernel(N_{\red})$. Together with 
Proposition \ref{prop:sim}, this yields 
$\kernel(M) = G \, \kernel(N_{\red}) G = G \, \kernel(N_{\red})$.
\qed
\end{proof}

\begin{remark}\label{rem:neut}
(i) If $N$ is reduced, then any homomorphism of algebraic monoids 
from $N$ to an algebraic group is trivial.

Indeed, let $\kappa : N \to \cG$ be such a homomorphism.
We may assume that $\kappa$ is the universal morphism 
$N \to H/K$ of Proposition \ref{prop:uni}, where $K$ is a normal 
subgroup scheme of $H$. Then the $G$-action on $N$ by conjugation
yields a $G$-action on $H/K$, compatible with the conjugation action
on $H$; thus, $K$ is a normal subgroup scheme of $G$. 
Moreover, the $H$-equivariant morphism $\kappa$ induces 
a $G$-equivariant morphism 
\[ \psi : M = G \times^H N \longrightarrow G \times^H H/K \cong G/K,
\quad (g,y)H \longmapsto (g,\kappa(y))H. \]
Also, $\psi(1)$ is the neutral element of $G/K$, 
since $\kappa(1)$ is the neutral element of $H/K$. 
It follows that $\psi(x y) = \psi(x) \psi(y)$
for all $x \in G$ and $y \in M$, and hence for all $x, y \in M$.
By Proposition \ref{prop:uni}, $\psi$ factors through $\varphi$, 
and hence $K = H$. 

We do not know if the analogous statement holds for $N_{\red}$
when $N$ is nonreduced. 

(ii) Let $M$ be a unit dense algebraic monoid, and $M^o$ 
its neutral component. Then $M^o$ is stable under the action 
of the unit group $G$ by conjugation on $M$; also, $M = G M^o$ 
by Proposition \ref{prop:neut}. Thus, $G \times^{G^o} M^o$ 
is an algebraic monoid, the disjoint union of the irreducible 
components of $M$. Moreover, the map
\[ \varphi : G \times^{G^o} M^o \longrightarrow M, \quad
(g,x)G^o \longmapsto gx \]
is a homomorphism of algebraic monoids, which is readily seen 
to be finite and birational. Hence $\varphi$ is an isomorphism
whenever $M$ is normal. 

For an arbitrary (unit dense) $M$, it follows that $E(M) = E(M^o)$. 
Indeed, $\varphi$ is surjective, and hence restricts to a surjection
$E(G \times^{G^o} M^o) \to E(M)$ by Corollary \ref{cor:idem}.
Also, $E(G \times^{G^o} M^o) = E(M^o)$, since the unique idempotent
of $G/G^o$ is the coset of $1_G$.
\end{remark}

\subsection{Structure of irreducible algebraic monoids}
\label{subsec:siam}

We begin this subsection by presenting some classical results 
on the structure of an arbitrary connected algebraic group $G$. 
By Chevalley's structure theorem, $G$ has a largest connected affine 
normal subgroup $G_{\aff}$; moreover, the quotient group 
$G/G_{\aff}$ is an abelian variety. In other words, $G$ sits in 
a unique exact sequence of connected algebraic groups
\[ 1 \longrightarrow G_{\aff} \longrightarrow G 
\stackrel{\alpha}{\longrightarrow} A \longrightarrow 1, \]
where $G_{\aff}$ is linear, and $A:= G/G_{\aff}$ is an abelian variety. 
This exact sequence is generally nonsplit; yet $G$ has a smallest
closed subgroup $H$ such that $\alpha\vert_H$ is surjective.
Moreover, $H$ is connected, contained in the center of $G$,
and satisfies $\cO(H) = k$. In fact, $H$ is the largest 
closed subgroup of $G$ satisfying the latter property,
which defines the class of \emph{anti-affine} algebraic groups; 
we denote $H$ by $G_{\ant}$. Finally, we have the 
\emph{Rosenlicht decomposition}: $G = G_{\aff} G_{\ant}$, 
and $(G_{\ant})_{\aff}$ is the connected neutral component
of the scheme-theoretic intersection $G_{\aff} \cap G_{\ant}$.
In other words, we have an isomorphism of algebraic groups
\[ G \cong (G_{\aff} \times G_{\ant})/(G_{\ant})_{\aff}, \]
and the quotient group scheme
$(G_{\aff} \cap G_{\ant})/(G_{\ant})_{\aff}$ is finite. We refer to
\cite{BSU} for an exposition of these results and of further 
developments.

We shall obtain a similar structure result 
for an arbitrary irreducible algebraic monoid; then the unit 
group is a connected algebraic group by Theorem \ref{thm:unit}.
Our starting point is the following:

\begin{proposition}\label{prop:hnaff}
Let $M$ be an irreducible algebraic monoid, $G$ its unit group, 
$\varphi : M \to \cG(M) = G/H$ the universal homomorphism to
an algebraic group, and $N$ the scheme-theoretic fiber of
$\varphi$ at $1$. Then $H$ and $N$ are affine.
\end{proposition}

\begin{proof}
\smartqed
Recall from Proposition \ref{prop:uni} that $H$ is the normal 
subgroup scheme of $G$ generated by $G_x$, where $x$ 
is an arbitrary point of $\kernel(M)$. Since $G_x$ is the isotropy 
subgroup scheme of a point for a faithful action of $G$ 
(the action on $M$ by left multiplication), it follows that $G_x$ is affine 
(see e.g. \cite[Cor.~2.1.9]{BSU}). The image of $G_x$ under the 
homomorphism $\alpha : G \to A$ is affine (as the image of an affine 
group scheme by a homomorphism of group schemes) and proper 
(as a subgroup scheme of the abelian variety $G/G_{\aff}$), 
hence finite. But $\alpha(G_x) = \alpha(H)$ by the definition 
of $H$ and the commutativity of $A$; hence $\alpha(H)$ is finite. 
Also, the kernel of the homomorphism $\alpha\vert_H$ is a subgroup
scheme of $G_{\aff}$, and hence is affine. Thus, the reduced scheme
$H_{\red}$ is an extension of a finite group by an affine algebraic
group, and hence is affine. Thus, so is $N_{\red}$ in view of
Theorem \ref{thm:aff} and of Proposition \ref{prop:nred}. 
It follows that $N$ is affine, by \cite[Exer.~III.3.1]{Hartshorne}).
\qed
\end{proof}

\begin{remark}\label{rem:triv}
If $\charc(k) = 0$, then $N$ is reduced and any 
homomorphism from $N$ to an algebraic group is trivial
by Remark \ref{rem:neut} (i). If in addition $N$ is
irreducible (e.g., if $M$ is normal), then $\kernel(N)$
is generated by the minimal idempotents. Indeed, the unit 
group $H$ of $N$ is generated by the conjugates of the
isotropy group $H_e$ for some idempotent 
$e \in \kernel(N)$, by Proposition \ref{prop:uni} (ii).
So the assertion follows from \cite[Thm.~2.1]{Putcha06}.
\end{remark}

A noteworthy consequence of Proposition \ref{prop:hnaff}
is the following:

\begin{corollary}\label{cor:qp}
Any irreducible algebraic monoid is quasiprojective.
\end{corollary}

\begin{proof}
\smartqed
With the notation of the above proposition, the morphism
$\varphi$ is affine, since $M = G \times^H N$ where $N$ is affine. 
Moreover, $G/H$ is quasiprojective since so is any 
algebraic group (see e.g. \cite[Prop.~3.1.1]{BSU}). 
Thus, $M$ is quasiprojective as well.
\qed
\end{proof}

Another consequence is a version of Chevalley's structure
theorem for an irreducible algebraic monoid; it generalizes
\cite[Thm.~1.1]{Brion-Rittatore}, where the monoid is assumed
to be normal.

\begin{theorem}\label{thm:maff}
Let $M$ be an irreducible algebraic monoid, $G$ its unit group,
and $M_{\aff}$ the closure of $G_{\aff}$ in $M$.

\begin{enumerate}
\item[{\rm (i)}]{$M_{\aff}$ is an irreducible affine algebraic monoid
with unit group $G_{\aff}$.}
\item[{\rm (ii)}]{The action of $G_{\aff}$ on $M_{\aff}$ extends to an
action of $G = G_{\aff} G_{\ant}$, where $G_{\ant}$ acts trivially.}
\item[{\rm (iii)}]{The natural map 
$G_{\ant} \times^{G_{\ant} \cap G_{\aff}} M_{\aff} \to 
G \times^{G_{\aff}} M_{\aff}$ 
is an isomorphism of irreducible algebraic monoids. 
Moreover, the natural map 
\[ \kappa: G \times^{G_{\aff}} M_{\aff} \to M \]
is a finite birational homomorphism of algebraic monoids.}
\item[{\rm (iv)}]{$E(M) = E(M_{\aff})$ and 
$\kernel(M) = G \, \kernel(M_{\aff})$.}
\item[{\rm (v)}]{ $M$ is normal if and only if $M_{\aff}$ is normal 
and $\kappa$ is an isomorphism. Then the assignment
$I \mapsto I \cap M_{\aff}$ defines a bijection between
the two-sided ideals of $M$ and those of $M_{\aff}$; the
inverse bijection is given by $J \mapsto GJ$. In particular,
$\kernel(M) \cap M_{\aff} = \kernel(M_{\aff})$.}
\end{enumerate}

\end{theorem}

\begin{proof}
\smartqed
(i) Clearly, $M_{\aff}$ is an irreducible submonoid of $M$, and
$G(M_{\aff})$ contains $G_{\aff}$ as an open subgroup. Since
$G(M_{\aff})$ is connected, it follows that $G(M_{\aff}) = G_{\aff}$. 
Hence $M_{\aff}$ is affine by Theorem \ref{thm:aff}.

(ii) follows readily from the Rosenlicht decomposition:
since $G_{\aff} \cap G_{\ant}$ is contained in the center
of $G_{\aff}$, its action on $M_{\aff}$ by conjugation
is trivial. Thus, the $G_{\aff}$-action by conjugation
on $M_{\aff}$ extends to an action of 
$G \cong (G_{\aff} \times G_{\ant})/(G_{\aff} \cap G_{\ant})$,
where $G_{\ant}$ acts trivially. 

(iii) The first assertion follows from the Rosenlicht 
decomposition again, since that decomposition yields
an isomorphism 
$G \cong G_{\ant} \times^{G_{\ant} \cap G_{\aff}} G_{\aff}$
of principal $G_{\aff}$-bundles over 
$G/G_{\aff} \cong G_{\ant}/(G_{\ant} \cap G_{\aff})$.
For the second assertion, note first that $\kappa$ restricts 
to the natural isomorphism $G \times^{G_{\aff}} G_{\aff} \to G$, 
and hence is a birational homomorphism of algebraic monoids.
To show that $\kappa$ is finite, we use the isomorphism 
$M \cong G \times^H N$ of Subsection \ref{subsec:ind}. Here
$H$ and $N$ are affine by Proposition \ref{prop:hnaff}; also, 
the natural map $G \times^{H_{\red}} N_{\red} \to M$ is finite and 
bijective by Proposition \ref{prop:nred}. It follows that 
the analogous map 
\[ \gamma : G \times^{H^o_{\red}} N_{\red}^o \longrightarrow M \] 
is finite and surjective. 
But $H^o_{\red} \subseteq G_{\aff}$ since $H$ is affine. Thus, 
\[ G \times^{H^o_{\red}} N_{\red}^o \cong 
G \times^{G_{\aff}} (G_{\aff} \times^{H^o_{\red}} N_{\red}^o). \]
Moreover, $N_{\red}^o \subseteq M_{\aff}$, since $N_{\red}^o$ 
is the closure in $M$ of $H^o_{\red} \subseteq G_{\aff}$; hence 
$G_{\aff} N_{\red}^o \subseteq M_{\aff}$. So $\gamma$ factors as
the natural map
\[ \beta: G \times^{G_{\aff}} (G_{\aff} \times^{H^o_{\red}} N_{\red}^o)
\to G \times^{G_{\aff}} M_{\aff} \]
(induced from the map
$\delta : G_{\aff} \times^{H^o_{\red}} N_{\red}^o \to M_{\aff}$),
followed by $\kappa$. Now $\delta$ is the restriction of $\gamma$
to a closed subvariety, and hence is finite; thus, its image
$G_{\aff} N_{\red}^o$ is closed in $M_{\aff}$. But 
$M_{\aff} = \overline{G_{\aff}}$, and hence $\delta$ is surjective.
Hence $\beta$ is finite and surjective.
Since $\gamma = \kappa \circ \beta$, it follows that $\kappa$
is finite and surjective as well.

(iv) Let $e \in E(M)$. By Corollary \ref{cor:idem}, we may lift $e$
to some idempotent $f$ of $G \times^{G_{\aff}} M_{\aff}$. 
Then the image of $f$ in $G/G_{\aff}$ is the neutral element, 
and hence $f \in M_{\aff}$ so that $e \in E(M_{\aff})$. The converse
is obvious.

Next, choose $e$ minimal. Then $\kernel(M) = GeG$ by Proposition
\ref{prop:sim}, and hence $\kernel(M) = G (G_{\aff} e G_{\aff})$
in view of the Rosenlicht decomposition. But 
$G_{\aff} e G_{\aff} = \kernel(M_{\aff})$ since $e$ is a minimal
idempotent of $G_{\aff}$. 

(v) Assume that $M$ is normal. By (iii) and Zariski's Main Theorem, 
it follows that $\kappa$ is an isomorphism.
In particular, $G \times^{G_{\aff}} M_{\aff}$ is normal. 
Since the natural morphism 
$G \times M_{\aff} \to G \times^{G_{\aff}} M_{\aff}$ is smooth,
it follows that $G \times M_{\aff}$ is normal (e.g., by Serre's 
criterion); hence so is $M_{\aff}$. The converse is straightforward.
This proves the first assertion.

The second assertion is proved by the argument of Proposition 
\ref{prop:small} (ii); note that any two-sided ideal of $M_{\aff}$
is stable under conjugation by $G$, in view of (ii) above.
\qed
\end{proof}

\begin{example}\label{ex:maff}
Let $n$ be a positive integer, $\mu_n$ the group scheme of
$n$th roots of unity, and $A$ an abelian variety containing
$\mu_n$ as a subgroup scheme (any ordinary elliptic curve will do).
As in Example \ref{ex:nonred}, let $N$ be the submonoid scheme
$(z^n = y^n)$ of $(\bA^2, \times)$, and $H$ the unit subgroup scheme
of $N$; then $H \cong \mu_n \times \bG_m$. Next, let 
$G := A \times \bG_m$; this is a connected commutative algebraic
group containing $H$ as a subgroup scheme. Finally, let
\[ M := G \times^H N = A \times^{\mu_n} N. \]
Then one checks that $M$ is an irreducible algebraic monoid with
unit group $G$. Clearly, $G_{\aff} = \bG_m$ and $A(G) = A$; also,
one checks that $M_{\aff} = (\bA^1, \times)$ and hence 
$G \times^{G_{\aff}} M_{\aff} \cong A \times \bA^1$. The morphism
$\kappa: G \times^{G_{\aff}} M_{\aff} \to M$ sends the closed
subscheme $\mu_n \times \{ 0 \}$ to $0$, and restricts to an 
isomorphism over the complement. So $M$ is obtained from
$A \times \bA^1$ by pinching $\mu_n \times \{ 0 \}$ to a point.
\end{example}

In view of Theorem \ref{thm:maff}, we may transfer information from 
affine algebraic monoids (about which much is known, see
\cite{Putcha88, Renner05}) to general ones. For example, the minimal
idempotents of any irreducible algebraic monoid are all conjugate
under the unit group, since this holds in the affine case by
\cite[Prop.~6.1, Cor.~6.8]{Putcha88}. Another noteworthy corollary 
is the following relation between the partial order on idempotents 
and limits of one-parameter subgroups:

\begin{corollary}\label{cor:1PS}
Let $(S,\mu)$ be an irreducible algebraic semigroup, and 
$e, f \in E(S)$. Then $e \leq f$ if and only if there exists 
a homomorphism of algebraic semigroups 
$\lambda : (\bA^1,\times) \to (S,\mu)$ 
such that $\lambda(0) = e$ and $\lambda(1) = f$.
\end{corollary}

\begin{proof}
\smartqed
The ``if'' implication is obvious (and holds in every algebraic
semigroup). For the converse, assume that $e \leq f$. Then 
$e \in f S f$ and the latter is an irreducible algebraic monoid. 
Thus, we may assume that $S$ itself is an irreducible algebraic 
monoid, and $f$ is the neutral element. In view of Theorem 
\ref{thm:maff}, we may further assume that $S$ is affine.
Then the assertion follows from 
\cite[Thm.~2.9, Thm.~2.10]{Putcha81}.
\qed
\end{proof}

\subsection{The Albanese morphism}
\label{subsec:tam}

By \cite[Sec.~4]{Serre58}, every irreducible variety $X$ admits 
a universal morphism to an abelian variety: the 
\emph{Albanese morphism},
\[ \alpha : X \longrightarrow A(X). \]
The group $A(X)$ is generated by the differences 
$\alpha(x) - \alpha(y)$, where $x,y \in X$.
Also, $X$ admits a universal rational map to an abelian variety:
the \emph{Albanese rational map},
\[ \alpha_{\rat} : X - \to A(X)_{\rat}. \]
The maps $\alpha$ and $\alpha_{\rat}$ are uniquely determined 
up to translations and isomorphisms of the algebraic group $A(X)$. 
Moreover, there exists a unique morphism
\[ \beta: A(X)_{\rat} \longrightarrow A(X) \]
such that $\alpha = \beta \circ \alpha_{\rat}$. The morphism $\beta$
is always surjective; when $X$ is nonsingular, it is an isomorphism. 
For an arbitrary $X$, we have $A(X)_{\rat} = A(U)$, where $U \subseteq X$ 
denotes the nonsingular locus; in particular, $\alpha_{\rat}$ is defined 
at any nonsingular point of $X$.

When $X$ is equipped with a base point $x$, we may assume that
$\alpha(x)$ is the origin of $A(X)$. If $X$ is nonsingular at
$x$, then we may further assume that $\alpha_{\rat}(x)$ is the 
origin of $A(X)_{\rat}$. Then $\alpha$ and $\alpha_{\rat}$ are unique
up to isomorphisms of algebraic groups.

Next, observe that the Albanese morphism of a connected linear
algebraic group $G$ is constant: indeed, $G$ is generated by
rational curves, and any morphism from such a curve 
to an abelian variety is constant. For a connected algebraic 
group $G$ (not necessarily linear), it follows that 
$\alpha = \alpha_{\rat}$ is the quotient homomorphism by the largest
connected affine subgroup $G_{\aff}$. This determines the Albanese 
rational map of an irreducible algebraic monoid $M$, which is 
just the Albanese morphism of its unit group. Some properties of
the Albanese morphism of $M$ are gathered in the following:

\begin{proposition}\label{prop:alb}
Let $M$ be an irreducible algebraic monoid with unit group $G$.
Then the map $\alpha: M \to A(M)$ is a homomorphism of algebraic 
monoids, and an affine morphism. Moreover, the map
$\beta: A(M)_{\rat} = A(G) \to A(M)$ is an isogeny. 
If $M$ is normal, then $\beta$ is an isomorphism. 
\end{proposition}

\begin{proof}
\smartqed
The monoid law $\mu : M \times M \to M$, $(1_M,1_M) \mapsto 1_M$
induces a morphism of varieties
$A(\mu) : A(M \times M) \to A(M)$, $0 \mapsto 0$.
Since $A(M \times M) = A(M) \times A(M)$, it follows that
$A(\mu)$ is a homomorphism; hence so is $\alpha$. In particular,
$\alpha$ factors through the universal homomorphism
$\varphi : M \to G/H$ of Proposition \ref{prop:uni}. Hence
$A(M) = A(G/H) = G/G_{\aff}H$, where $G_{\aff}H$ is a normal subgroup
scheme of $G$ such that the quotient
$G_{\aff}H/G_{\aff} \cong H/(H \cap G_{\aff})$ is finite.
Write $M = G \times^H N$ as in Proposition \ref{prop:hnaff}; then
\[ M \cong G \times^{G_{\aff}H} (G_{\aff}H \times^H N) \]
and this identifies $\alpha$ with the natural map to $G/G_{\aff}H$, 
with fiber $G_{\aff}H \times^H N$. But that fiber is affine, 
since so are $N$ and $G_{\aff}H/H \cong G_{\aff}/(G_{\aff} \cap H)$. 
It follows that the morphism $\alpha$ is affine. Also,
$\beta$ is identified with the natural homomorphism
$G/G_{\aff} \to G/G_{\aff}H$; hence the kernel of $\beta$ is
isomorphic to $G_{\aff}H/H$, a finite group scheme.

If $M$ is normal, then $M \cong G \times^{G_{\aff}} M_{\aff}$
by Theorem \ref{thm:maff}. Thus, the natural map $M \to G/G_{\aff}$
is the Albanese morphism.
\qed
\end{proof}

Consider for instance the monoid $M$ constructed in Example 
\ref{ex:maff}. Then $A(M) \cong A/\mu_n$ and $A(G) \cong A$;
this identifies $\beta$ to the quotient morphism $A \to A/\mu_n$.

Returning to our general setting, recall that every irreducible 
algebraic monoid may be viewed 
as an equivariant embedding of its unit group. For an arbitrary
equivariant embedding $X$ of a connected algebraic group $G$, 
we may again identify $A(X)_{\rat}$ with $A(G)$; when $X$ is normal,
we still have $A(X) = A(X)_{\rat}$ as a consequence of 
\cite[Thm.~3]{Brion10}. But the morphism $\alpha$ is generally 
nonaffine, and the finiteness of $\beta$ is an open question
in this setting.

We now characterize algebraic monoids among equivariant embeddings:

\begin{theorem}\label{thm:monemb}
Let $X$ be an equivariant embedding of a connected algebraic group
$G$. Then $X$ has a structure of algebraic monoid with unit group 
$G$ if and only if the Albanese morphism $\alpha :X \to A(X)$ is affine.
\end{theorem}

\begin{proof}
\smartqed
In view of Proposition \ref{prop:alb}, it suffices to show that $X$
is an algebraic monoid if $\alpha$ is affine. Note that $\alpha$
is $G \times G$-equivariant for the given action of $G \times G$ 
on $X$, and a transitive action on $A(X)$. It follows that
$A(X) \cong (G \times G)/(K \times K) \diag(G)$ for a unique normal
subgroup scheme $K$ of $G$; then $A(X) \cong G/K$ equivariantly
for the left (or right) action of $G$. Moreover, $\alpha$ is a
fiber bundle of the form 
\[ G \times G \times^{(K \times K) \diag(G)} Y \longrightarrow 
(G \times G)/(K \times K) \diag(G), \]
where $Y$ is a scheme equipped with an action of 
$(K \times K)\diag(G)$; for the left (or right) $G$-action, 
this yields the fiber bundle $G \times^K Y \to G/K$.
Since $\alpha$ is affine, so is $Y$. 
Also, $Y$ meets the open orbit $G \cong (G \times G)/\diag(G)$ along 
a dense open subscheme isomorphic to $K$, where $K \times K$ acts by 
left and right multiplication, and $\diag(G)$ by conjugation.
Thus, the group scheme $K$ is quasi-affine, and hence is affine.

We now show that the group law $\mu_K : K \times K \to K$ extends 
to a morphism $\mu_Y : Y \times Y \to Y$, by following the argument 
of \cite[Prop.~1]{Rittatore98}. The left action $K \times Y \to Y$ and the 
right action $Y \times K \to Y$ restrict both to $\mu_K$ on $K \times K$,
and hence yield a morphism $(K \times Y) \cup (Y \times K) \to Y$ . 
Since $Y$ is affine, it suffices to show the equality 
\[ \cO((K \times Y) \cup (Y \times K)) = \cO(Y \times Y). \]
But 
$\cO(Y \times Y) = \cO(Y) \otimes \cO(Y) \subseteq 
\cO(K) \otimes \cO(K) = \cO(K \times K)$, since $K$ is dense in $Y$.
Moreover, 
\[ \cO((K \times Y) \cup (Y \times K)) = 
(\cO(K) \otimes \cO(Y)) \cap (\cO(Y) \otimes \cO(K)), \]
where the intersection is considered in $\cO(K) \otimes \cO(K)$. 
Now for any vector space $V$ and subspace $W$, we easily
obtain the equality 
$(W \otimes V) \cap (V \otimes W) = W \otimes W$ as subspaces
of $V \otimes V$. When applied to $\cO(Y) \subseteq \cO(K)$, 
this yields the desired equality.

Since $\mu_Y$ is associative on the dense subscheme $K$, it is associative 
everywhere; likewise, $\mu_Y$ admits $1_K$ as a neutral element.
Thus, $\mu_Y$ is an algebraic monoid law on $Y$. We may now form the 
induced monoid $G \times^K Y$ as in Subsection \ref{subsec:ind}, 
to get the desired structure on~$X$.
\qed
\end{proof}

\subsection{Algebraic semigroups and monoids over perfect fields}
\label{subsec:asmopf}

In this subsection, we extend most of the above results to the 
setting of algebraic semigroups and monoids defined over 
a perfect field. We use the terminology and results of 
\cite{Springer}, especially Chapter 11 which discusses basic 
rationality results on varieties. 

Let $F$ be a subfield of the algebraically closed field $k$. 
We assume that $F$ is \emph{perfect}, i.e., every algebraic 
extension of $F$ is separable; we denote by $\bar{F}$ 
the algebraic closure of $F$ in $k$, and by $\Gamma$ the 
Galois group of $\bar{F}$ over $F$.

We say that an algebraic semigroup $(S,\mu)$ (over $k$) is 
\emph{defined over $F$}, or an \emph{algebraic $F$-semigroup},
if $S$ is an $F$-variety and the morphism $\mu$ is defined over 
$F$. Then the set of $\bar{F}$-points, $S(\bar{F})$, is 
a subsemigroup of $S$ equipped with an action of $\Gamma$ 
by semigroup automorphisms, and the fixed point subset 
$S(\bar{F})^{\Gamma}$ is the semigroup of $F$-points, $S(F)$.

Note that an algebraic $F$-semigroup may well have no $F$-point;
for example, an $F$-variety without $F$-point equipped with
the trivial semigroup law $\mu_l$ or $\mu_r$. But this is the
only obstruction to the existence of $F$-idempotents, as shown 
by the following:
 
\begin{proposition}\label{prop:idemF}
Let $(S,\mu)$ be an algebraic $F$-semigroup. 

\begin{enumerate}
\item[{\rm (i)}]{$E(S)$ and $\kernel(S)$ (viewed as closed subsets of 
$S$) are defined over $F$.}
\item[{\rm (ii)}]{If $S$ is commutative, then its smallest idempotent
is defined over $F$.}
\item[{\rm (iii)}]{If $S$ has an $F$-point, then it has an idempotent 
$F$-point.}
\end{enumerate}

\end{proposition}

\begin{proof}
\smartqed
(i) Clearly, $E(S)$ and $\kernel(S)$ are defined over $\bar{F}$ 
and their sets of $\bar{F}$-points are stable under the action of 
$\Gamma$ on $S(\bar{F})$. Thus, $E(S)$ and $\kernel(S)$ 
are defined over $F$ by \cite[Prop.~11.2.8(i)]{Springer}. 

(ii) is proved similarly.

(iii) Let $x \in S(F)$ and denote by $\langle x \rangle$ the
closure in $S$ of the set $\{ x^n, \; n \geq 1 \}$. Then
$\langle x \rangle$ is a closed commutative subsemigroup of $S$, 
defined over $F$ by \cite[Lem.~11.2.4]{Springer}. 
In view of (ii), $\langle x \rangle$ contains an idempotent 
defined over $F$.
\qed
\end{proof}

We do not know if any algebraic $F$-semigroup $S$ has a minimal
idempotent defined over $F$. This holds if $S$ is irreducible,
as we will see in Proposition \ref{prop:irrF}. First, we record
two rationality results on algebraic monoids:

\begin{proposition}\label{prop:monoF}
Let $(M,\mu,1_M)$ be an algebraic monoid with unit group
$G$ and neutral component $M^o$. If $M$ and $\mu$ are defined
over $F$, then so are $1_M$, $G$ and $M^o$. Moreover, the inverse
map $\iota : G \to G$ is defined over $F$.
\end{proposition}

\begin{proof}
\smartqed
Observe that $1_M$ is the unique point $x \in M$ such that
$x y = y x = y$ for all $y \in M(\bar{F})$ (since $M(\bar{F})$
is dense in $M$). It follows that $1_M \in M(\bar{F})$; also,
$1_M$ is $\Gamma$-invariant by uniqueness. Thus, $1_M \in M(F)$.

The assertion on $G$ follows from \cite[Prop.~11.2.8(ii)]{Springer}. 
It implies that $G^o$ is defined over $F$ by 
[loc.~cit., Prop.~12.1.1]. Since $M^o$ is the closure of 
$G^o$ in $M$, it is also defined over $F$ in view of 
[loc.~cit., Prop.~11.2.8(i)].

It remains to show that $\iota$ is defined over $F$;
equivalently, its graph is an $F$-subvariety of $G \times G$.
But this graph equals
\[ \{ (x,y) \in G \times G ~\vert~ x y = 1_M \} = 
\mu_G^{-1}(1_M), \]
where $\mu_G : G \times G \to G$ denotes the restriction of $\mu$,
and $\mu_G^{-1}(1_M)$ stands for the set-theoretic fiber. Moreover,
this fiber is defined over $F$ in view of 
\cite[Cor.~11.2.14]{Springer}.
\qed
\end{proof}

\begin{proposition}\label{prop:irrF}
Let $(M,\mu,1_M)$ be an irreducible algebraic monoid with unit 
group $G$. If $(M,\mu)$ is defined over $F$, 
then the universal homomorphism to an algebraic group, 
$\varphi: M \to \cG(M)$, is defined over $F$ as well. Moreover, 
$G_{\aff}$ and $M_{\aff}$ are defined over $F$.
\end{proposition}

\begin{proof}
\smartqed
The first assertion follows from the uniqueness of $\varphi$
by a standard argument of Galois descent, see 
\cite[Chap.~V, \S 4]{Serre59}. The (well-known) assertion on
$G_{\aff}$ is proved similarly; it implies the assertion on 
$M_{\aff}$ by \cite[Prop.~11.2.8(i)]{Springer}.
\qed
\end{proof}

Returning to algebraic semigroups, we obtain the promised:

\begin{proposition}\label{prop:minF}
Let $(S,\mu)$ be an irreducible algebraic $F$-semigroup. If $S$ 
has an $F$-point, then some minimal idempotent of $S$ is
defined over $F$.
\end{proposition}

\begin{proof}
\smartqed
By Proposition \ref{prop:idemF}, we may choose $e \in E(S(F))$.
Then $eSe$ is a closed irreducible submonoid of $S$, and is defined
over $S$ in view of \cite[Prop.~11.2.8(i)]{Springer} again. Moreover,
any minimal idempotent of $eSe$ is a minimal idempotent of $S$.
So we may assume that $S$ is an irreducible monoid, $M$. 
In view of Theorem \ref{thm:maff} and Proposition \ref{prop:irrF}, 
we may further assume that $M$ is affine. Then the unit group of $M$ 
contains a maximal torus $T$ defined over $F$, 
by Proposition \ref{prop:monoF} and 
\cite[Thm.~13.3.6, Rem.~13.3.7]{Springer}. The closure $\bar{T}$
of $T$ in $M$ is defined over $F$, and meets $\kernel(M)$
in view of \cite[Cor.~6.10]{Putcha88}. So the (set-theoretic) 
intersection $N := \bar{T}\cap \kernel(M)$ is a commutative algebraic 
semigroup, defined over $F$ by \cite[Thm.~11.2.13]{Springer}. 
Now applying Proposition \ref{prop:idemF} to $N$ yields the 
desired idempotent.
\qed
\end{proof}

\begin{remark}\label{rem:nonper}
The above observations leave open all the rationality questions for
an algebraic semigroup $S$ over a field $F$, not necessarily perfect.
In fact, $S$ has an idempotent $F$-point if it has an $F$-point, 
as follows from the main result of \cite{Brion-Renner}. 
But some results do not extend to this setting: for example, 
the kernel of an algebraic $F$-monoid may not be defined over $F$, 
as shown by a variant of the standard example of a linear algebraic 
$F$-group whose unipotent radical is not defined over $F$ (see 
\cite[Exp.~XVII, 6.4.a)]{SGA3} or \cite[12.1.6]{Springer};
specifically, replace the multiplicative group $\bG_m$ with
the monoid $(\bA^1, \times)$ in the construction of this example). 
Also, note that Chevalley's structure theorem fails over any 
imperfect field $F$ (see \cite[Exp.~XVII, App.~III, Cor.]{SGA3}, 
and \cite{Totaro} for recent developments). Thus, $G_{\aff}$ 
may not be defined over $F$ with the notation and assumptions 
of Proposition \ref{prop:irrF}. Yet the Albanese morphism still 
exists for any $F$-variety equipped with an $F$-point 
(see \cite[App.~A]{Wittenberg}) and hence for any algebraic 
$F$-semigroup equipped with an $F$-idempotent.
\end{remark}

\section{Algebraic semigroup structures on certain varieties}
\label{sec:asscv}

\subsection{Abelian varieties}
\label{subsec:av}

In this subsection, we begin by describing all the algebraic 
semigroup laws on an abelian variety. Then we apply the result to 
the study of the Albanese morphism of an irreducible algebraic 
semigroup. 

\begin{proposition}\label{prop:abel}
Let $A$ be an abelian variety with group law denoted additively, 
$\mu$ an algebraic semigroup law on $A$, and $e$ an idempotent of 
$(A,\mu)$; choose $e$ as the neutral element of $(A,+)$. 

\begin{enumerate}
\item[{\rm (i)}]{There exists a unique decomposition of algebraic semigroups
\[ (A,\mu) = 
(A_0,\mu_0) \times (A_l,\mu_l) \times (A_r,\mu_r) \times (B,+) \]
where $A_0$, $A_l$, $A_r$ and $B$ are abelian varieties, and
$\mu_0$ (resp. $\mu_l, \mu_r$) the trivial semigroup law on $A_0$ 
(resp. $A_l$, $A_r$) defined in Example \ref{ex:semi} (i).}
\item[{\rm (ii)}]{The corresponding projection $\varphi : A \to B$
is the universal homomorphism of $(A,\mu)$ to an algebraic group.
Moreover, we have 
$E(S) = \{ e \} \times A_l \times A_r \times \{ e \}$
and $\kernel(S) = \{ e \} \times A_l \times A_r \times B$.}
\end{enumerate}

\end{proposition}

\begin{proof}
\smartqed
(i) By \cite[Chap.~II, \S 4, Cor.~1]{Mumford}, the morphism 
$\mu : A \times A \to A$ satisfies 
\[ \mu(x,y) = \varphi(x) + \psi(y) + x_0, \]
where $x_0 \in A$ and $\varphi$, $\psi$ are endomorphisms of the
algebraic group $A$. Since $\mu(e,e)= e$ and 
$\varphi(e) = \psi(e) = e$, we have $x_0 = e$, i.e., 
$\mu(x,y) = \varphi(x) + \psi(y)$. 
Now the associativity of $\mu$ is equivalent to the equality
\[ \varphi \circ \varphi (x) + \varphi \circ \psi(y) + \psi(z) 
= \varphi (x) + \psi \circ \varphi(y) + \psi \circ \psi(z), \]
that is, to the equalities
\[ \varphi \circ \varphi = \varphi, \quad 
\varphi \circ \psi = \psi \circ \varphi, \quad
\psi \circ \psi = \psi. \]
This easily yields the desired decomposition, where 
$A_0 := \Kern(\varphi) \cap \Kern(\psi)$,
$A_l := \Ima(\varphi) \cap \Kern(\psi)$,
$A_r := \Kern(\varphi) \cap \Ima(\psi)$,
and $B := \Ima(\varphi) \cap \Ima(\psi)$,
so that $\varphi$ (resp. $\psi$) is the projection of $A$ to
$A_l \times B$ (resp. $A_r \times B$). 
The uniqueness of this decomposition follows from that of 
$\varphi$ and $\psi$. 

(ii) Let $\gamma : (A,\mu) \to \cG$ be a homomorphism to an
algebraic group. Then the image of $\gamma$ is a complete
irreducible variety, and hence generates an abelian subvariety 
of $\cG$. Thus, we may assume that $\cG$ is an abelian variety, 
with group law also denoted additively. As above, we have 
$\gamma(x) = \pi(x) + x'_0$, where $\pi : A \to \cG$ is a 
homomorphism of algebraic groups and $x'_0 \in \cG$. Since 
$\gamma(e)$ is idempotent, we obtain $x'_0 = 0$, i.e., 
$\gamma: (A,+) \to \cG$ is also a homomorphism. It follows readily
that $\gamma$ sends $A_0 \times A_l \times A_r \times \{ e \}$
to $0$. So $\gamma$ factors as $\varphi$ followed by 
a unique homomorphism $\gamma': B \to \cG$.
This proves the assertion on $\varphi$; those on $E(S)$
and $\kernel(S)$ are easily checked.
\qed
\end{proof}

\begin{proposition}\label{prop:albs}
Let $(S,\mu)$ be an irreducible algebraic semigroup, $e \in E(S)$, 
and $\alpha : S \to A(S)$ the Albanese morphism; assume that
$\alpha(e) = 0$.

\begin{enumerate}
\item[{\rm (i)}]{There exists a unique algebraic semigroup law $A(\mu)$ 
on $A(S)$ such that $\alpha$ is a homomorphism.}
\item[{\rm (ii)}]{Let $\varphi : (A(S),A(\mu)) \to B(S)$ be the universal 
homomorphism to an algebraic group. Then the map 
$eSe \to B(S)$, $x \mapsto \varphi(\alpha(x))$ is the Albanese 
morphism of $eSe$.}
\end{enumerate}

\end{proposition}

\begin{proof}
\smartqed
(i) follows from the functorial properties of the Albanese 
morphism (see \cite[Sec.~2]{Serre58}) by arguing as in the 
beginning of the proof of Proposition \ref{prop:alb}.

(ii) Consider the inclusions $eSe \subseteq eS \subseteq S$. Each of 
them admits a retraction, $x \mapsto xe$ (resp. $x \mapsto ex$).
Thus, the corresponding morphisms of Albanese varieties,
$A(eSe) \to A(eS) \to A(S)$, 
also admit retractions, and hence are closed immersions. So we may
identify $A(eSe)$ with the subgroup of $A(S)$ generated by the
differences $\alpha(exe) - \alpha(eye)$, where $x,y \in S$. But
$\alpha(exe) = A(\mu)(\alpha(e),A(\mu)(\alpha(x),\alpha(e)))$ and
$\alpha(e)$ is of course an idempotent of $(A(S),A(\mu))$. 
Hence $\alpha(e) = (e,a_l, a_r, e)$ in the decomposition of 
Proposition \ref{prop:abel}. Using that decomposition, we obtain
$\alpha(exe) = (e, a_l, a_r, b(x))$, where $b(x)$ denotes 
the projection of $\alpha(x)$ to $B(S)$. As a consequence,
$\alpha(exe) - \alpha(eye) = (e,e,e,b(x) - b(y))$; this yields 
the desired identification of $A(eSe)$ to $B(S)$.
\qed
\end{proof}

Combined with Proposition \ref{prop:alb}, the above result yields:

\begin{corollary}\label{cor:albs}
Let $S$ be an irreducible algebraic semigroup. 

\begin{enumerate}
\item[{\rm (i)}]{All the maximal submonoids of $S$ have the same Albanese 
variety, and all the maximal subgroups have isogenous Albanese 
varieties.}
\item[{\rm (ii)}]{The irreducible monoid $eSe$ is affine for all 
$e \in E(S)$ if $eSe$ is affine for some $e \in E(S)$.}
\end{enumerate}

\end{corollary}

\begin{remark}\label{rem:albs}
(i) With the notation and assumptions of Proposition 
\ref{prop:albs}, the morphism $\varphi \circ \alpha : S \to B(S)$ 
is the universal homomorphism to an abelian variety. 
Also, recall from Remark \ref{rem:uni} that there exists 
a universal homomorphism to an algebraic group, 
$\psi : S \to \cG(S)$, and that $\cG(S)$ is connected. 
It follows that $B(S)$ is the Albanese variety of $\cG(S)$.

(ii) Consider an irreducible algebraic semigroup $(S,\mu)$ 
and its rational Albanese map $\alpha_{\rat} : S - \to A(S)_{\rat}$.
If the image of $\mu: S \times S \to S$ meets the domain of 
definition of $\alpha_{\rat}$, then there exists a unique 
algebraic semigroup structure $A(\mu)$ on $A(S)_{\rat}$ such that 
$\alpha_{\rat}$ is a `rational homomorphism', i.e.,
$\alpha_{\rat}(\mu(x,y)) = A(\mu)(\alpha_{\rat}(x), \alpha_{\rat}(y))$ 
whenever $\alpha_{\rat}$ is defined at $x, y \in S$ and at $\mu(x,y)$
(as can be checked by the argument of Proposition \ref{prop:albs}). 
But this does not hold for an arbitrary $(S,\mu)$; for example, 
if $S \subseteq \bA^3$ is the affine cone over an elliptic curve
$E \subseteq \bP^2$ and if $\mu = \mu_0$. Here $0$, the origin of 
$\bA^3$, is the unique singular point of $S$, and $\alpha_{\rat}$
is the natural map $S \setminus \{ 0  \} \to E$.
\end{remark}

\subsection{Irreducible curves}
\label{subsec:ic}

In this subsection, we classify the irreducible algebraic semigroups
of dimension $1$; those having a nontrivial law (as defined in 
Example \ref{ex:semi} (i)) turn out to be algebraic monoids. 

Such semigroups include of course the connected algebraic groups 
of dimension $1$, presented in Example \ref{ex:mono} (iv). 
We now construct further examples: let $(a_1, \ldots, a_n)$ 
be a strictly increasing sequence of positive integers having 
no nontrivial common divisor, and consider the map
\[ \varphi : \bA^1 \longrightarrow \bA^n, \quad 
x \longmapsto (x^{a_1},\ldots,x^{a_n}). \]
Then $\varphi$ is a homomorphism of algebraic monoids, where
$\bA^1$ and $\bA^n$ are equipped with pointwise multiplication.
Also, one checks that the morphism $\varphi$ is finite; 
hence its image is a closed submonoid of $\bA^n$, containing 
the origin as its zero element. We denote this monoid by 
$M(a_1,\ldots,a_n)$, and call it an \emph{affine monomial curve}; 
it only depends on the abstract submonoid of $(\bZ,+)$ generated 
by $a_1,\ldots,a_n$. One may check that $\varphi$ restricts 
to an isomorphism 
$\bA^1 \setminus \{ 0 \} \stackrel{\cong}{\longrightarrow} 
M(a_1,\ldots, a_n) \setminus \{ 0 \}$;
also, $M(a_1,\ldots,a_n)$ is singular at the origin
unless $\varphi$ is an isomorphism, i.e., unless $a_1 = 1$.

\begin{theorem}\label{thm:curve}
Let $S$ be an irreducible curve, and $\mu$ a nontrivial algebraic 
semigroup structure on $S$. Then $(S,\mu)$ is either an algebraic 
group or an affine monomial curve.
\end{theorem}

\begin{proof}
\smartqed
As the arguments are somewhat long and indirect, we divide them
into four steps.

\medskip

\noindent
{\bf Step 1:} we show that every idempotent of $S$ is either 
a neutral or a zero element.

Let $e \in E(S)$. Since $Se$ is a closed irreducible subvariety 
of $S$, it is either the whole $S$ or a single point; in the 
latter case, $Se = \{ e \}$. Thus, one of the following cases occurs:

\begin{enumerate}
\item[(i)]{$Se = eS = S$. Then any $x \in S$ satisfies $xe = ex = x$,
i.e., $e$ is the neutral element.}
\item[(ii)]{$Se = \{ e \}$ and $eS = S$. Then for any $x,y \in S$, we have
$xe = e$ and $ey = y$. Thus, $xy = xey = ey = y$. So $\mu = \mu_r$
in the notation of Example \ref{ex:semi} (i), a contradiction 
since $\mu$ is assumed to be nontrivial.}
\item[(iii)]{$eS = \{ e \}$ and $Se = S$. This case is excluded similarly.}
\item[(iv)]{$Se = eS = \{ e \}$. Then $e$ is the zero element of $S$.}
\end{enumerate}

\medskip

\noindent
{\bf Step 2:} we show that if $S$ is complete, then it is 
an elliptic curve.

For this, we first reduce to the case where $S$ has a zero
element. Otherwise, $S$ has a neutral element by Step 1. Hence
$S$ is a monoid with unit group $G$ being $\bG_a$, 
$\bG_m$ or an elliptic curve, in view of the classification of
connected algebraic groups of dimension $1$. In the latter case,
$G$ is complete and hence $G = S$. On the other hand, 
if $G = \bG_a$ or $\bG_m$, then $S \setminus G$ is a nonempty 
closed subsemigroup of $S$ in view of Proposition \ref{prop:unit}. 
Hence $S \setminus G$ contains an idempotent, which must be 
the zero element of $S$ by Step 1. This yields the desired reduction.

The semigroup law $\mu : S \times S \to S$ sends $S \times \{ 0 \}$
to the point $0$. By the rigidity lemma (see e.g.
\cite[Chap.~II, \S 4]{Mumford}), it follows that 
$\mu(x,y) = \varphi(y)$ for some morphism $\varphi : S \to S$.
The associativity of $\mu$ yields
\[ \varphi(z) = (xy)z = x(yz) = \varphi(yz)= \varphi(\varphi(z)) \]
for all $x,y,z \in S$; hence $\varphi$ is a retraction to its
image. Since $S$ is an irreducible curve, either $\varphi = \id$
or the image of $\varphi$ consists of a single point $x$. 
In the former case, $\mu = \mu_r$, whereas $\mu = \mu_x$ 
in the latter case. Thus, the law $\mu$ is trivial, a contradiction.

\medskip

\noindent
{\bf Step 3:} we show that if $S$ is an affine monoid, then
it is isomorphic to $\bG_a$, $\bG_m$ or an affine monomial curve.

We may view $S$ as an equivariant embedding of its unit group
$G$, and that group is either $\bG_a$ or $\bG_m$. Since 
$\bG_a \cong \bA^1$ as a variety, any affine equivariant embedding 
of $\bG_a$ is $\bG_a$ itself. So we may assume that $G = \bG_m$. 
Then the coordinate ring $\cO(S)$ is a subalgebra of 
$\cO(\bG_m) = k[x,x^{-1}]$, stable under 
the natural action of $\bG_m$. It follows that $\cO(S)$ has a basis
consisting of Laurent monomials, and hence that
\[ \cO(S) = \bigoplus_{n \in \cM} x^n, \]
where $\cM$ is a submonoid of $(\bZ,+)$. Moreover, since $\bG_m$
is open in $S$, the fraction field of $\cO(S)$ is the field
of rational functions $k(x)$; it follows that $\cM$ generates the
group $\bZ$. Thus, either $\cM = \bZ$ or $\cM$ is generated by
finitely many integers, all of the same sign and having no nontrivial 
common divisor. In the former case, $S = \bG_m$; in the latter case, 
$S$ is an affine monomial curve.

\medskip

\noindent
{\bf Step 4:} in view of Step 2, we may assume that the irreducible
curve $S$ is noncomplete, and hence is affine. Then it suffices 
to show that $S$ has a nonzero idempotent: then $S$ is an algebraic 
monoid by Step 1, and we conclude by Step 3. We may further assume 
that $S$ is nonsingular: indeed, by the nontriviality assumption, 
the semigroup law $\mu : S \times S \to S$ is dominant. 
Using Proposition \ref{prop:norm}, it follows that the normalization 
$\tilde S$ (an irreducible nonsingular curve) has a compatible algebraic 
semigroup structure; then the image of a nonzero idempotent 
of $\tilde S$ is a nonzero idempotent of $S$.

So we assume that $S$ is an affine irreducible nonsingular semigroup 
of dimension $1$, having a zero element $0$, and show that 
$S$ has a neutral element. We use the ``right regular representation'' 
of $S$, i.e., its action on the coordinate ring $\cO(S)$ by right 
multiplication; specifically, an arbitrary point $x \in S$ acts on 
$\cO(S)$ by sending a regular function $f$ on $S$ to the regular 
function $x \cdot f : y \mapsto f(yx)$. This yields a map 
\[ \varphi : S \longrightarrow \End(\cO(S)), 
\quad x \longmapsto  x \cdot f \]
which is readily seen to be a homomorphism of abstract
semigroups. Moreover, the action of $S$ on $\cO(S)$ stabilizes
the maximal ideal $\fm$ of $0$, and all its powers $\fm^n$. 
This defines compatible homomorphisms of abstract semigroups
\[ \varphi_n : S \longmapsto \End(\fm/\fm^n) \quad (n \geq 1). \]
Since $S$ is a nonsingular curve, we have compatible 
isomorphisms of $k$-algebras 
\[ \fm/\fm^n \cong k[t]/t^n k[t], \]
where $t$ denotes a generator of the maximal ideal $\fm \cO_{S,0}$
of the local ring $\cO_{S,0}$; the right-hand side is the algebra 
of truncated polynomials at the order $n$. Thus, an endomorphism
$\gamma$ of $\fm/\fm^n$ is uniquely determined by $\gamma(\bar{t})$,
where $\bar{t}$ denotes the image of $t$ mod $t^n k[t]$. 
Moreover, the assignment $\gamma \mapsto \gamma(\bar{t})$
yields compatible isomorphisms of abstract semigroups
\[ \End(\fm/\fm^n) \stackrel{\cong}{\longrightarrow} 
t k[t]/t^n k[t], \]
where the semigroup law on the right-hand side is
the composition of truncated polynomials. Thus, we obtain 
compatible homomorphisms of abstract semigroups
\[ \psi_n : S \longrightarrow t k[t]/t^n k[t]. \]
Clearly, the right-hand side is an algebraic semigroup. Moreover,
$\psi_n$ is a morphism: indeed, for any $f \in \cO(S)$, we have
$f(yx) = \sum_{i \in I} f_i(x) g_i(y)$ for some finite collection
of functions $f_i, g_i \in \cO(S)$ (since the semigroup law is
a morphism). In other words, $x \cdot f = \sum_{i \in I} f_i(x) g_i$.
Thus, the matrix coefficients of the action of $x$ in $\cO(S)/\fm^n$,
and hence in $\fm/\fm^n$, are regular functions of~$x$.

We claim that there exists $n\geq 1$ such that $\psi_n \neq 0$.
Otherwise, we have $\varphi_n(x) = 0$ for all $n \geq 1$ and 
all $x \in S$. Since $\bigcap_n \fm^n = \{ 0 \}$, it follows that 
$\varphi(x)$ sends $\fm$ to $0$. But $\cO(S) = k \oplus \fm$,
where the line $k$ of constant functions is fixed pointwise
by $\varphi(x)$. Hence $\varphi(x) = \varphi(0)$ for all $x$, 
i.e., $f(yx) = f(0)$ for all $f \in \cO(S)$ and all $x,y \in S$. 
Thus, $yx = 0$, i.e., $\mu = \mu_0$; a contradiction.

Now let $n$ be the smallest integer such that $\psi_n \neq 0$.
Then $\psi_n$ sends $S$ to the quotient $t^{n-1} k[t]/t^nk[t]$, 
i.e., to the semigroup of endomorphisms of the algebra 
$k[t]/t^n k[t]$ given by $\bar{t} \mapsto c\bar{t}^{n-1}$, where $c \in k$.
If $n \geq 3$, then the composition of any two such endomorphisms
is $0$, and hence $\psi_n(xy) = 0$ for all $x,y \in S$. Thus,
$xy$ belongs to the fiber of $\psi_n$ at $0$, a finite set
containing $0$. Since $S$ is irreducible, it follows that 
$xy = 0$, i.e., $\mu = \mu_0$; a contradiction. Thus, we must have 
$n = 2$, and we obtain a nonconstant morphism
$\psi = \psi_2: S \to \bA^1$, where the semigroup law on $\bA^1$ 
is the multiplication. The image of $\psi$ contains $0$ and a nonempty
open subset $U$ of the unit group $\bG_m$. Then $U U = \bG_m$
and hence $\psi$ is surjective. By Proposition \ref{prop:idem}, 
it follows that there exists an idempotent $e \in S$ such that 
$\psi(e) = 1$. Then $e$ is the desired nonzero idempotent.
\qed
\end{proof}

\begin{remark}\label{rem:cur}
One may also deduce the above theorem from the description of
algebraic semigroup structures on abelian varieties 
(Proposition \ref{prop:abel}), when the irreducible curve $S$ 
is assumed to be nonsingular and nonrational. Then the Albanese
morphism of $S$ is a locally closed embedding in its Jacobian
variety $A$. It follows that $A$ has no trivial summand $A_0$, 
$A_l$ or $A_r$ (otherwise, the projection to that summand is 
constant since $\mu$ is nontrivial; as the differences of points 
of $S$ generate the group $A$, this yields a contradiction). 
In other words, the inclusion of $S$ into $A$ is a homomorphism 
for a suitable choice of the origin of $A$. This implies 
that $S = A$, and we conclude that $S$ is an elliptic curve 
equipped with its group law.
\end{remark}

\subsection{Complete irreducible varieties}
\label{subsec:civ}

In this subsection, we obtain a description of all complete irreducible 
algebraic semigroups, analogous to that of the kernels of algebraic
semigroups presented in Proposition \ref{prop:SeS}:

\begin{theorem}\label{thm:cisg}
There is a bijective correspondence between the following objects:

\begin{itemize}
\item{the triples $(S,\mu,e)$, where $S$ is a complete 
irreducible variety, $\mu$ an algebraic semigroup structure 
on $S$, and $e$ an idempotent of $(S,\mu)$,}
\item{the tuples $(X,Y,G,\iota,\rho,x_o,y_o)$, where $X$ 
(resp. $Y$) is a complete irreducible variety equipped with 
a base point $x_o$ (resp. $y_o$), $G$ is an abelian variety, 
$\iota: X \times G \times Y \to S$ is a closed immersion, and
$\rho : S \to X \times G \times Y$ a retraction of $\iota$.}
\end{itemize}

This correspondence assigns to any such tuple, the algebraic 
semigroup structure $\nu$ on $X \times G \times Y$ defined by
\[ \nu((x,g,y), (x',g',y')) := (x, g g', y') \]
and then the algebraic semigroup structure $\mu$ on $S$ defined by
\[ \mu(s,s') :=  \iota(\nu(\rho(s), \rho(s'))). \]
The idempotent is $e := \iota(x_o,1_G,y_o)$. Moreover, $\iota$
and $\rho$ are homomorphisms of algebraic semigroups. 
\end{theorem}

The inverse correspondence will be constructed at the end of the
proof. We begin that proof with three preliminary results.

\begin{lemma}\label{lem:retract}
Let $\varphi : X \to Y$ be a morphism of varieties, where $X$ is complete 
and irreducible; assume that $\varphi$ has a section (for example, $\varphi$
is a retraction of $X$ to a subvariety $Y$). 
Then $Y$ is complete and irreducible as well. Moreover, the map 
$\varphi^{\#} : \cO_Y \to \varphi_*(\cO_X)$ is an isomorphism;
in particular, the fibers of $\varphi$ are connected.
\end{lemma}

\begin{proof}
\smartqed
Note that $\varphi$ is surjective, since it admits a section. This readily yields 
the first assertion.

Next, consider the Stein factorization of $\varphi$ as the composition
\[ X \stackrel{\varphi'}{\longrightarrow} X' 
\stackrel{\psi}{\longrightarrow} Y, \]
where $\varphi'$ is the natural morphism to the Spec of the sheaf 
of $\cO_Y$-algebras $\varphi_*(\cO_X)$, and $\psi$ is finite (see 
\cite[Cor.~III.11.5]{Hartshorne}). Then $\varphi'$ is surjective, and hence 
$X'$ is a complete irreducible variety.  Also, given a section $\sigma$ 
of $\varphi$, the map $\varphi' \circ \sigma$ is a section 
of $\psi$. In view of the irreducibility of $X'$ and the finiteness of $\psi$, 
it follows that $\psi$ is an isomorphism; this yields the second assertion.
\qed
\end{proof}

\begin{lemma}\label{lem:idem}
Let $S$ be a complete irreducible algebraic semigroup, and 
$e$ an idempotent of $S$. Then $x y = x e y$ for all $x,y \in S$.
\end{lemma}

\begin{proof}
\smartqed
Recall that the map $\varphi : S \to eS$, $x \mapsto ex$ 
is a retraction. Thus, its fibers are connected by Lemma 
\ref{lem:retract}. Let $F$ be a (set-theoretic) fiber.
Then the morphism $\mu : S \times S \to S$, 
$(x,y) \mapsto xy$ sends $\{ e \} \times F$ to a point. 
By the rigidity lemma (see e.g. \cite[Chap.~II, \S 4]{Mumford}),
$\mu(\{ x \} \times F)$ consists of a single point for any $x \in S$.
Thus, the map $y \mapsto xy$ is constant on the fibers of $\varphi$.
Since $\varphi(y) = \varphi(ey)$ for all $y \in S$, this yields the
statement.
\qed
\end{proof}

\begin{lemma}\label{lem:abel}
Keep the assumptions of the above lemma.

\begin{enumerate}
\item[{\rm (i)}]{The closed submonoid $e S e$ of $S$ is an abelian variety.}
\item[{\rm (ii)}]{The map $\varphi : S \to eSe$, $x \mapsto exe$ 
is a retraction of algebraic semigroups.}
\item[{\rm (iii)}]{The above map $\varphi$ is the universal homomorphism 
to an algebraic group.}
\end{enumerate}

\end{lemma}

\begin{proof}
\smartqed
(i) By Proposition \ref{prop:unit} (iii), it suffices to show 
that $e$ is the unique idempotent of $eSe$. But if $f \in E(eSe)$, 
then $xy = xfy$ for all $x,y \in S$, by Lemma \ref{lem:idem}.
Taking $x = y = e$ yields $e = efe = f$. 

(ii) By Lemma \ref{lem:idem} again, we have 
$e xy e = e x e y e = (exe)(eye)$ for all $x,y \in S$.

(iii) Let $\cG$ be an algebraic group and let $\psi : S \to \cG$ 
be a homomorphism of algebraic semigroups. Then $\psi(e) = 1$ 
and hence $\psi(x) = \psi(exe)$ for all $x \in S$. Thus, 
$\psi$ factors uniquely as the homomorphism $\varphi$ followed
by some homomorphism of algebraic groups $eSe \to \cG$.
\qed
\end{proof}

\begin{remark}\label{rem:cisg}
By Lemma \ref{lem:abel}, every idempotent $e$ of a complete irreducible 
algebraic semigroup $(S,\mu)$ is minimal. Moreover, by Lemma 
\ref{lem:idem}, the image of the morphism $\mu$ is exactly the kernel 
of $S$; this is a simple algebraic semigroup in view of Proposition 
\ref{prop:SeS}. One may thus deduce part of Theorem \ref{thm:cisg} 
from the structure of simple algebraic semigroups presented in 
Remark \ref{rem:simple} (i). Yet we will provide a direct, 
self-contained proof by adapting the arguments of Proposition 
\ref{prop:SeS}.
\end{remark}

\medskip

\noindent
{\sc Proof of Theorem \ref{thm:cisg}.}

One readily checks that the map $\nu$ (resp. $\mu$) as in the 
statement yields an algebraic semigroup structure on 
$X \times G \times Y$ (resp. on $S$); compare with 
Example \ref{ex:semi} (ii).

Conversely, given $(S,\mu,e)$ as in the statement, consider
\[ X := {_e{}S e}, \quad G := e S e, \quad Y := e S_e \]
with the notation of Remark \ref{rem:idem} (ii).
Then $G$ is an abelian variety by Lemma \ref{lem:abel}.
Let $ \iota: X \times G \times Y \to S$ denote the multiplication 
map: $\iota(x,g,y) = xgy$. Finally, define a map
$\rho : S \to S \times G \times S$ by 
\[ \rho(s) = (s (ese)^{-1}, ese, (ese)^{-1}s). \]
Then $s (ese)^{-1} \in X$, since $es (ese)^{-1} = ese (ese)^{-1} = e$
and $s (ese)^{-1} e = s (ese)^{-1}$. Likewise, $(ese)^{-1} s \in Y$.
So the image of $\rho$ is contained in $X \times G \times Y$.

We claim that $\rho \circ \iota$ is the identity of 
$X \times G \times Y$. Indeed, 
$(\rho \circ \iota)(x,g,y) = \rho(xgy)$. Moreover, $exgye = g$
so that 
\[ \rho(xgy) = (xgyg^{-1}, g, g^{-1}gy). \]
Now $xgyg^{-1} = xgyeg^{-1} = xgeg^{-1} = xe = x$ and likewise,
$g^{-1}xgy = y$. This proves the claim.

By that claim, $\iota$ is a closed immersion, and $\rho$ a retraction
of $\iota$. Also, we have for any $x,x' \in X$, $g,g' \in G$ and
$y,y' \in Y$:
\[ x g y x' g' y' = x g y e x' g' y' = x g e x' g' y' = x g e g' y'. \]
In other words, $\iota$ is a homomorphism of algebraic semigroups, 
where $X \times G \times Y$ is given the semigroup structure $\nu$
as in the statement.

We next claim that $\rho$ is a homomorphism of algebraic semigroups
as well. Indeed, 
\[ \rho(ss') = (ss' (ess'e)^{-1}, ess'e, (ess'e)^{-1}ss')\]
and hence, using Lemma \ref{lem:abel},
\[ \rho(ss') = 
(ss' (es'e)^{-1} (ese)^{-1}, eses'e, (es'e)^{-1} (ese)^{-1} ss'). \]
Moreover,
\[ ss' (es'e)^{-1} (ese)^{-1} = ses'e (es'e)^{-1} (ese)^{-1}
= se (ese)^{-1} = s (ese)^{-1} \] 
by Lemma \ref{lem:idem}, and likewise 
$(es'e)^{-1} (ese)^{-1} ss'= (es'e)^{-1} s'$. Thus, 
\[ \rho(ss') = (s(ese)^{-1}, eses'e, (es'e)^{-1}s')
= \nu(\rho(s), \rho(s')) \]
as required. 

Finally, we claim that $ss' = \iota (\nu(\rho(s), \rho(s')))$.
Indeed, the right-hand side equals
\[ s (ese)^{-1} eses'e (es'e)^{-1}s' = s e s' = s s' \]
in view of Lemma \ref{lem:idem} again.

\begin{remark}\label{rem:cc}
(i) The description of algebraic semigroup laws on a given 
abelian variety $A$ (Proposition \ref{prop:abel}) may of course
be deduced from Theorem \ref{thm:cisg}: with the notation of 
that theorem, the inclusion $\iota$ and retraction $\rho$ yield 
a decomposition $A \cong A_0 \times A_l \times A_r \times B$,
where $A_l := X$, $A_r := Y$, $B := G$ and $A_0$ denotes the
fiber of $\rho$ at $0$. Yet the original proof of Proposition
\ref{prop:abel} is simpler and more direct. 

(ii) As a direct consequence of Theorem \ref{thm:cisg}, 
every algebraic semigroup law on a complete irreducible curve 
is either trivial or the group law of an elliptic curve. 
This yields an alternative proof of part of the classification 
of irreducible algebraic semigroups of dimension $1$
(Theorem \ref{thm:curve}); but in fact, both arguments make 
a similar use of the rigidity lemma.

(iii) As another consequence of Theorem \ref{thm:cisg}, for any
complete irreducible algebraic semigroup $(S,\mu)$, the closed subset 
$E(S)$ of idempotents is an irreducible subsemigroup. Indeed,
choosing $e \in E(S)$, we have with the notation of that theorem 
\[ E(S) = \iota(X \times \{ 1_G \} \times Y) \cong X \times Y. \]
Moreover, $\mu(\iota(x,1_G,y), \iota(x',1_G,y')) = \iota(x,1_G,y')$
with an obvious notation.

(iv) In fact, some of the ingredients of Theorem \ref{thm:cisg} 
only depend of $(S,\mu)$, but not  of the choice of $e \in E(S)$.
Specifically, note first that the projections $\varphi: E(S) \to X$, 
$\psi : E(S) \to Y$ are independent of $e$. Indeed, let $f \in E(S)$ 
and write $f = \iota(x,1_G,y)$. Then 
$f E(S) = \iota(\{ (x,1_G) \} \times Y)$ and hence $f E(S)$ is the fiber
of $\varphi$ at $f$; likewise, the fiber of $\psi$ at $f$ is $E(S) f$.

As seen in Remark \ref{rem:cisg}, $\kernel(S) = S e S$ is isomorphic
to $X \times G \times Y$ via $\iota$.
The resulting projection $\gamma : \kernel(S) \to G$ is the universal
homomorphism to an algebraic group by Lemma \ref{lem:abel}, and 
hence is also independent of $e$; its fiber at $1_G$ is $E(S)$. 
In particular, the algebraic group $G = \cG(S)$ is independent of 
$e$. Note however that $G = e S e$, viewed as a subgroup of $S$,
does depend of the choice of the idempotent $e$. Indeed,
$e S e = \iota(\{ x_o \} \times G \times \{ y_o \})$ with the notation
of Theorem \ref{thm:cisg}, while 
$f S f = \iota(\{ x \} \times G \times \{ y \})$ for $f$ as above.

The map $\rho: S \to X \times G \times Y$ satisfies
$\iota \circ \rho = \rho_e$, where $\rho_e : S \to \kernel(S)$ 
denotes the retraction $s \mapsto s (ese)^{-1} s$ of Proposition 
\ref{prop:SeS}. We check that $\rho_e$ is independent of $e$ 
(this also follows from Proposition \ref{prop:rig} below).
Let $f\in E(S)$, $s \in S$, and write $f = \iota(x,1_G,y)$, 
$\rho(s) = (x_s, g_s, y_s)$. Then $f s f = \iota(x, g_s, y)$ 
and hence $(fsf)^{-1} = \iota(x, g_s^{-1}, y)$. Thus,
\[ \rho_f(s) = s (fsf)^{-1} s = \iota(x_s, g_s, y_s) 
= \rho_e(s). \]

Finally, consider the action of the abelian variety $G$ on 
$\kernel(S) \cong X \times G \times Y$ via translation on the
second factor:
\[ g' \cdot \iota(x,g,y) := \iota(x,gg',y). \]
We check that this action lifts to an action of $G$ on $S$
such that $\rho : S \to \kernel(S)$ is equivariant.
For any $s,s' \in S$, define
\[ s' \cdot s := s (ese)^{-1} s' s. \]
Then we have
\[ s' \cdot s = s (ese)^{-1} \, es'e \, e s 
= s (ese)^{-1} \, ese \, es'e \, (ese)^{-1}s. \]
It follows that $s' \cdot s = es'e \cdot s = es'e \cdot \rho(s)$. 
Moreover,  $s' \cdot S = \kernel(S)$ and the endomorphism 
$s \mapsto s' \cdot s$ of $\kernel(S)$ is just the translation
by $e s' e \in G$ on $G = e S e$; we have
\[ s' \cdot s_1 s_2 = s_1 s' s_2 \] 
for all $s_1, s_2 \in S$. Also, one may check as above that 
$s' \cdot s$ is independent of the choice of $e$.

(v) Theorem \ref{thm:cisg} extends readily to those irreducible 
algebraic semigroups that are defined over a perfect subfield $F$ 
of $k$, and that have an $F$-point; indeed, this implies the existence 
of an idempotent $F$-point by Proposition~\ref{prop:idemF}.

Likewise, the results of Subsections \ref{subsec:av} and
\ref{subsec:ic} extend readily to the setting of perfect fields.
In view of Theorem \ref{thm:curve}, 
every nontrivial algebraic semigroup law $\mu$ 
on an irreducible curve $S$ is commutative; by Proposition 
\ref{prop:idemF} again, it follows that  $S$ has an idempotent
$F$-point whenever $S$ and $\mu$ are defined over~$F$.
\end{remark}

\subsection{Rigidity}
\label{subsec:rig}

In this subsection, we obtain two rigidity results (both possibly known,
but for which we could not locate adequate references) and we apply
them to the study of endomorphisms of complete varieties.

Our first result is a scheme-theoretic version of a classical rigidity 
lemma for irreducible varieties (see \cite[Lem.~1.15]{Debarre}; 
further versions can be found in \cite[Prop.~6.1]{MFK94}).

\begin{lemma}\label{lem:rig}
Let $f: X \to Y$ and $g : X \to Z$ be morphisms of schemes of finite
type over $k$, satisfying the following assumptions:

\begin{enumerate}
\item[{\rm (i)}]{$f$ is proper and the map 
$f^{\#} : \cO_Y \to f_*(\cO_X)$ is an isomorphism.}
\item[{\rm (ii)}]{There exists a $k$-rational point $y_o \in Y$ such that 
$g$ maps the scheme-theoretic fiber $f^{-1}(y_o)$ to a single point.}
\item[{\rm (iii)}]{$f$ has a section, $s$.}
\item[{\rm (iv)}]{$X$ is irreducible.}
\end{enumerate}

Then $g$ factors through $f$; specifically, $g = h \circ f$, 
where $h := g \circ s$. 
\end{lemma}

\begin{proof}
\smartqed
We first treat the case where $y_o$ is the unique closed
point of $Y$. We claim that $X$ is the unique open neighborhood 
of $f^{-1}(y_o)$. Indeed, given such a neighborhood $U$ with
complement $F := X \setminus U$, the image $f(F)$ is closed,
since $f$ is proper. If $F$ is nonempty, then $f(F)$ contains
$y_o$, a contradiction.

Let $z_o\in Z$ be the point $g(f^{-1}(y_o))$, and choose an open
affine neighborhood $W$ of $z_o$ in $Z$. Then $g^{-1}(W) = X$ by the 
claim together with (iii); thus, we may assume that $Z = W$ is affine. 
Then $g$ is uniquely determined by the homomorphism of algebras
$g^{\#}: \cO(Z) \to \cO(X)$.
But the analogous map $f^{\#}: \cO(Y) \to \cO(X)$ is an isomorphism
in view of (iv). Thus, there exists a morphism $h' : Y \to Z$ 
such that $g = h' \circ f$. Then 
$h = g \circ s = h' \circ f \circ s = h'$; this completes the
proof in that case (note that the assumptions (i) and (ii) suffice
to conclude that $g$ factors through $f$).

Next, we treat the general case. The scheme 
$Y' := \Spec(\cO_{Y,y_o})$ has a unique closed point $y'_o$ and 
comes with a flat morphism $\psi : Y' \to Y$, $y'_o \mapsto y_o$. 
Moreover, $X' := X \times _Y Y'$ is equipped with morphisms
$f' : X' \to Y'$, $g' = g \circ p_1: X' \to Z$ that satisfy (i) 
(since taking the direct image commutes with flat base extension, 
see \cite[Prop.~III.9.3]{Hartshorne}) and (ii). Also, note
that $f'$ has a section $s'$ given by the morphism 
$(s \circ \psi) \times \id : Y' \to X \times Y'$.
By the preceding step, we thus have $g' = h' \circ f'$, where
$h' := g' \circ s'$. It follows that there exists an open
neighborhood $V$ of $y_0$ in $Y$ such that $g = h \circ f$ over 
$f^{-1}(V)$.

We now consider the largest subscheme $W$ of $X$ over which 
$g = h \circ f$, i.e., $W$ is the preimage of the diagonal 
in $Z \times Z$ under the morphism $g \times (h \circ f)$.
Then $W$ is closed in $X$ and contains $f^{-1}(V)$. Since
$X$ is irreducible, it follows that $W = X$.
\qed
\end{proof}

\begin{remark}\label{rem:rig}
The assertion of Lemma \ref{lem:rig} still holds under the 
assumptions (ii), (iii) and the following (weaker but more
technical) versions of (i), (iv):

\begin{enumerate}
\item[(i)']{$f$ is proper, and for any irreducible component $Y'$ of $Y$, 
the scheme-theoretic preimage $X':= f^{-1}(Y')$ is an
irreducible component of $X$. Moreover, the map
$f'^{\#}: \cO_{Y'} \to f'_*(\cO_{X'})$ is an isomorphism, where
$f': X' \to Y'$ denotes the restriction of $f$.}
\item[(iv)']{$X$ is connected.}
\end{enumerate}

Indeed, let $Y_o$ be an irreducible component of $Y$ 
containing $y_o$; then $X_o := f^{-1}(Y_o)$
is an irreducible component of $X$. Moreover, the restrictions
$f_o : X_o \to Y_o$, $g_o : X_o \to Z$, and $s_o : Y_o \to X_o$ 
satisfy the assumptions of Lemma \ref{lem:rig}. 
By that lemma, it follows that $g_o = g_o \circ s_o \circ f_o$,
i.e., $g = h \circ f$ on $X_o$. 
In particular, $g$ maps the scheme-theoretic fiber of $f$ at
any point of $Y_o$ to a single point.

Next, let $Y_1$ be an irreducible component of $Y$ intersecting
$Y_o$. Then again, $X_1 := f^{-1}(Y_1)$ is an irreducible
component of $X$; moreover, the restrictions
$f_1 : X_1 \to Y_1$, $g_1 : X_1 \to Z$ and $s_1 : Y_1 \to X_1$ 
satisfy the assumptions of the above lemma, for any point $y_1$ of 
$Y_o \cap Y_1$. Thus, $g = h \circ f$ on $X_o \cup X_1$.           
Iterating this argument completes the proof in view of the
connectedness of~$X$.
\end{remark}

As a first application of the above lemma and remark, 
we present a rigidity result for retractions; further applications 
will be obtained in the next subsection.

\begin{proposition}\label{prop:rig}
Let $X$ be a complete irreducible variety, and $\varphi$ 
a retraction of $X$ to a subvariety $Y$. Let $T$ be a connected 
scheme of finite type over $k$, equipped with a $k$-rational point 
$t_o$, and let $\Phi : X \times T \to X$ be a morphism such that
the morphism $\Phi_{t_o} : X \to X$, $x \mapsto \Phi(x,t_o)$
equals $\varphi$.  

\begin{enumerate}
\item[{\rm (i)}]{There exists a unique morphism 
$\Psi: Y \times T \to X$ such that 
$\Phi(x,t) = \Psi(\varphi(x),t)$ on $X \times T$.}
\item[{\rm (ii)}]{If $\Phi$ is a family of retractions to $Y$
(i.e., $\Phi(y,t) = y$ on $Y \times T$), then $\Phi$ is 
constant (i.e., $\Phi(x,t) = \varphi(x)$ on $X \times T$).}
\end{enumerate}

\end{proposition}

\begin{proof}
\smartqed
Consider the morphisms 
\[ f : X \times T \longrightarrow Y \times T, \quad 
(x,t) \mapsto (\varphi(x),t),\]
\[ g : X \times T \longrightarrow X \times T, \quad 
(x,t) \mapsto (\Phi(x,t),t).\]
Then the assumption (i)' of Remark \ref{rem:rig} holds, since 
$\varphi_*(\cO_X) = \cO_Y$ in view of Lemma 
\ref{lem:retract}.  Also, the assumption (ii) of Lemma \ref{lem:rig}
holds for any point $(y,t_o)$, where $y \in Y$, and the assumption
(iii) of that lemma holds with $s$ being the inclusion 
of $Y \times T_o$ in $X \times T_o$. Finally, the assumption 
(iv)' of Remark \ref{rem:rig} is satisfied,
since $T$ is connected. By that remark, we thus have 
$g = g \circ s \circ f$ on $X \times T$. Hence there exists 
a unique morphism $\Psi : Y \times T \to X$ such that 
$\Phi(x,t) = \Psi(\varphi(x),t)$ on $X \times T$, namely,
$\Psi(y,t) := \Phi(y,t)$. If $\Phi$ is a family of retractions,
then we get that $\Phi(x,t) = \varphi(x)$ on $X \times T$.  
\qed
\end{proof}

\begin{remark}\label{rem:mod}
The preceding result has a nice interpretation when $X$ is 
projective. Then there exists a quasiprojective $k$-scheme, 
$\End(X)$, which represents the endomorphism functor of $X$, 
i.e., for any noetherian $k$-scheme $T$, the set of $T$-points 
$\End(X)(T)$ is naturally identified with the set of endomorphisms 
of $X \times T$ over $T$; equivalently,
\[ \End(X)(T) = \Hom(X \times T,X). \]
Moreover, each connected component of $\End(X)$ is of finite type.
These results hold, more generally, for the similarly defined
functor $\Hom(X,Y)$ of morphisms from a projective scheme $X$ 
to another projective scheme $Y$ (see \cite[p.~21]{Grothendieck}). 
The composition of morphisms yields a morphism of Hom functors, 
and hence of Hom schemes by Yoneda's lemma. In particular, 
$\End(X)$ is a monoid scheme; its idempotent $k$-points are exactly 
the retractions with source $X$. 

Returning to the setting of an irreducible projective variety $X$
together with a retraction $\varphi: X \to Y$, we may identify 
$\varphi$ with the idempotent endomorphism $e$ of $X$ with image 
$Y$. Now Proposition \ref{prop:rig} yields that the connected 
component of $e$ in $\End(X)$ is isomorphic to the connected 
component of the inclusion $Y \to X$ in $\Hom(Y,X)$, 
by assigning to any
$\phi \in \Hom(Y,X)(T) = \Hom(Y \times T,X)$, the composition 
$\psi \circ (\varphi \times \id) \in \Hom(X \times T,X)$.
Moreover, this isomorphism identifies the connected component
of $e$ in 
\[ \End(E)_e := 
\{ \Phi \in \End(X) ~\vert~ \Phi \circ e = e \} \]
to the (reduced) point $e$.
\end{remark}

Next, we obtain our second rigidity result:

\begin{lemma}\label{lem:end}
Let $X$ be a complete variety, $T$ a connected scheme 
of finite type over $k$, and 
\[ \Phi : X \times T \longrightarrow X \times T,
\quad (x,t) \longmapsto (\varphi(x,t),t) \]
an endomorphism of $X \times T$ over $T$. Assume that
$T$ has a point $t_o$ such that 
$\Phi_{t_o}: X \to X$, $x \mapsto \varphi(x,t_o)$ is an 
automorphism. Then $\Phi$ is an automorphism.
\end{lemma}

\begin{proof}
\smartqed
Note that $\Phi$ is proper, as the composition of the closed 
immersion $X \times T \to X \times X \times T$, 
$(x,t) \mapsto (x,\varphi(x,t),t)$ and of the projection
$X \times X \times T \to X \times T$, $(x,y,t) \mapsto (y,t)$.

We now show that the fibers of $\Phi$ are finite. 
Assuming the contrary, we may find a complete irreducible
curve $C \subseteq X$ and a point $t_1 \in T$ such that 
$\varphi : X \times T \to X$ sends $C \times \{ t_1 \}$ to a point. 
By the rigidity lemma, it follows that the restriction
of $\varphi$ to $C \times T$ factors through the projection
$C \times T \to T$. Taking $t = t_o$, we get a contradiction.
  
The morphism $\Phi$ is finite, since it is proper and its fibers 
are finite; it is also surjective, since for any $t\in T$,
the map $\Phi_t : X \to X$, $x \mapsto \varphi(x,t)$ is a finite 
endomorphism of $X$ and hence is surjective. 

We now claim that $\Phi$ restricts to an automorphism of 
$X \times V$, for some open neighborhood $V$ of $t_o$ in $T$. 
This claim is proved in \cite[Lem.~I.1.10.1]{Kollar}; 
we recall the argument for completeness. Since $\Phi$ is proper, 
the sheaf $\Phi_*(\cO_{X \times T})$ is coherent; it is also flat 
over $T$, since $\Phi$  lifts the identity of $T$. 
Moreover, the map 
$\Phi^{\#} : \cO_{X \times T} \to \Phi_*(\cO_{X \times T})$
induces an isomorphism 
$\Phi^{\#}_{t_o} : \cO_X \to (\Phi_{t_o})_*(\cO_X)$.
In view of a version of Nakayama's lemma 
(see \cite[Prop.~I.7.4.1]{Kollar}), it follows that $\Phi^{\#}$ 
is an isomorphism over a neighborhood of $t_o$. This yields
the claim.

By that claim, the points $t \in T$ such $\Phi_t$ is an 
isomorphism form an open subset of $T$. Since $T$ is connected,
it suffices to show that this subset is closed.
For this, we may assume that $T$ is an irreducible curve; 
replacing $T$ with its normalization, we may further assume that $T$ 
is nonsingular. By shrinking $T$, we may finally assume that it has 
a point $s$ such that $\varphi_t$ is an automorphism for all 
$t \in T \setminus \{ s \}$; we have to show that $\varphi_s$ 
is an automorphism as well.

If $X$ is normal, then so is $X \times T$; moreover, the above
endomorphism $\Phi$ is finite and birational, and hence an 
automorphism. Thus, every $\varphi_t$ is an automorphism.

For an arbitrary $X$, consider the normalization
$\eta : \tilde X \to X$. Then $\Phi$ lifts to an endomorphism
$\tilde\Phi : \tilde X \times T \to \tilde X \times T$, which is
an automorphism by the above step. In particular, 
$\varphi_s$ lifts to an automorphism $\tilde{\varphi}_s$
of $\tilde X$. We have a commutative diagram
\[ \CD
\tilde X @>{\tilde{\varphi}_s}>> \tilde X \\
@V{\eta}VV @V{\eta}VV \\
X @>{\varphi_s}>> X \\
\endCD \]
and hence a commutative diagram of morphisms of sheaves
\[ \CD
\cO_X @>>> (\varphi_s)_*(\cO_X) \\
@VVV @VVV \\
\eta_*(\cO_{\tilde X}) @>>> \eta_*(\tilde{\varphi}_s)_*(\cO_{\tilde X}). \\
\endCD \]
Moreover, the bottom horizontal arrow in the latter diagram is 
the identity (as $(\tilde \varphi_s)_*(\cO_{\tilde X}) = \cO_{\tilde X}$), 
and the other maps are all injective. Thus,
$\cO_X \subseteq (\varphi_s)_*(\cO_X) \subseteq \eta_*(\cO_{\tilde X})$,
and hence the iterates $(\varphi^n_s)_*(\cO_X)$ form an increasing
sequence of subsheaves of $\eta_*(\cO_{\tilde X})$. As the latter
sheaf is coherent, we get 
\[ (\varphi^n_s)_*(\cO_X) = (\varphi^{n+1}_s)_*(\cO_X)
\quad (n \gg 0). \]
Since $\varphi_s$ is finite and surjective, it follows that 
$\cO_X = (\varphi_s)_*(\cO_X)$ and hence that $\varphi_s$
is an isomorphism.
\qed
\end{proof}

A noteworthy consequence of Lemma \ref{lem:end} is the following:

\begin{corollary}\label{cor:abel}
{\rm (i)} Let $M$ be a complete algebraic monoid. Then $G(M)$ is 
a union of connected components of $M$. In particular, if $M$
is connected then it is an abelian variety.

{\rm (ii)} Let $S$ be a complete algebraic semigroup and let $e,f$
be distinct idempotents such that $e \leq f$. Then $e$ and $f$ belong
to distinct connected components of $S$. In particular, if $S$ is
connected then every idempotent is minimal.
\end{corollary}

\begin{proof}
\smartqed
(i) Let $T$ be a connected component of $M$ containing a unit $t_o$. 
Applying Lemma \ref{lem:end} to the morphism
$M \times T \to M \times T$, $(x,t) \mapsto (xt, t)$,
we see that the map $x \mapsto xt$ 
is an isomorphism for any $t \in T$. Likewise, the map 
$x \mapsto tx$ is an isomorphism as well. Thus, 
$t$ has a left and a right inverse in $M$, and hence is a unit.
So $T$ is contained in $G(M)$.

Alternatively, we may deduce the statement from Theorem 
\ref{thm:curve}: indeed, $G(M)$ contains no subgroup
isomorphic to $\bG_a$ or $\bG_m$, since the latter do not
occur as unit groups of complete irreducible monoids.
By Chevalley's structure theorem, it follows that the
reduced neutral component of $G(M)$ is an abelian variety.
Thus, $G(M)$ is complete, and hence closed in $M$.
But $G(M)$ is open in $M$, hence the assertion follows. 

(ii)  Assume that $e$ and $f$ belong to the same connected
component $T$ of $S$. Then $T$ is a closed subsemigroup, 
and hence we may assume that $S$ is connected. Now 
$fSf$ is a complete connected algebraic monoid, and hence 
an abelian variety. It follows that $e = f$, a contradiction.
\qed
\end{proof}

\begin{remark}\label{rem:mum}
(i) Like for Proposition \ref{prop:rig}, the statement of 
Lemma \ref{lem:end} has a nice interpretation 
when $X$ is projective. Then its functor of automorphisms
is represented by an open subscheme $\Aut(X)$ of $\End(X)$  
(see \cite[p.~21]{Grothendieck}); in fact, $\Aut(X)$ is 
the unit group scheme of the monoid scheme $\End(X)$. 
Now Lemma \ref{lem:end} implies that $\Aut(X)$ 
is also closed in $\End(X)$. In other words, $\Aut(X)$ 
is a union of connected components of $\End(X)$.

For an arbitrary complete variety $X$, the automorphism  
functor defined as above is still represented by a group
scheme $\Aut(X)$; moreover, each connected component of $\Aut(X)$
is of finite type (see \cite[Thm.~3.7]{Matsumura-Oort} for these
results). We do not know if $\End(X)$ is representable in this 
generality; yet the above interpretation of Lemma
\ref{lem:end} still makes sense in terms of functors.

(ii) Let $X$ and $T$ be complete varieties, where $T$
is irreducible, and let $\mu : X \times T \to X$ be
a morphism such that $\mu(x,t_o) = x$
for some $t_o \in T$ and all $x \in X$. Then by 
Lemma \ref{lem:end}, the map $\mu_t: x \mapsto \mu(x,t)$ 
is an automorphism for any $t \in T$. This yields a morphism
of schemes
\[ \varphi : T \longrightarrow \Aut(X), \quad t \longmapsto \mu_t \]
such that $\varphi(t_o)$ is the identity. Hence $\varphi$ sends $T$ 
to the neutral component $\Aut^o(X)$. Consider the subgroup $G$ of
$\Aut^o(X)$ generated by the image of $T$; then $G$ is closed 
and connected by \cite[Prop.~II.5.4.6]{Demazure-Gabriel}, 
and hence is an abelian variety. In loose words, the morphism
$\mu$ arises from an action of an abelian variety on $X$.

(iii) Let $X$ be a complete irreducible variety, and
$\mu : X \times X \to X$ a morphism such that 
$\mu(x,x_o) = \mu(x_o,x) = x$ for some $x_o \in X$ and all $x \in X$. 
Then the above morphism $\varphi : X \to \Aut^o(X)$ 
satisfies $\varphi(x)(x_o) = x$, and hence is a closed immersion;
we thus identify $X$ to its image in $\Aut^o(X)$. As seen above,
$X$ generates an abelian subvariety $G$ of $\Aut^o(X)$. The 
natural action of $G$ on $X$ is transitive, since the orbit
$G x_o$ contains $X x_o = X$. Thus, $X$ itself is an abelian variety 
on which $G$ acts by translations. Moreover, since $G x_o = X x_o$, 
we have $G = X G_{x_o}$, where $G_{x_o}$ denotes the isotropy subgroup 
scheme of $x_o$. As $G$ is commutative and acts faithfully and 
transitively on $X$, this isotropy subgroup scheme is trivial, 
i.e., $G = X$. In conclusion, $X$ is an abelian variety with group 
law $\mu$ and neutral element $x_o$. This result is due to Mumford
and Ramanujam, see \cite[Chap.~II, \S 4, Appendix]{Mumford}.
\end{remark}

\subsection{Families of semigroup laws}
\label{subsec:fam}

\begin{definition}\label{def:fam}
Let $S$ be a variety, and $T$ a $k$-scheme. 
A \emph{family of semigroup laws on $S$ parameterized by $T$}
is a morphism $\mu: S \times S \times T \to S$
such that the associativity condition
\[ \mu(s, \mu(s', s'', t), t ) = \mu( \mu(s, s', t), s'', t) \]
holds on $S \times S \times S \times T$. 
\end{definition}

Such a family yields a structure of semigroup scheme on 
$S \times T$ over $T$: to any scheme $T'$ equipped with a
morphism $\theta : T' \to T$, one associates the (abstract)
semigroup consisting of all morphisms $\sigma : T' \to S$,
 equipped with the law $\mu_{\theta}$ defined by
\[ \mu_{\theta}(\sigma,\sigma')= 
\mu(\sigma,\sigma',\theta). \]
In particular, any $k$-rational point $t_o$ of $T$
yields an algebraic semigroup structure on $S$, 
\[ \mu_{t_o} : S \times S \longrightarrow  S, 
\quad (s, s') \longmapsto \mu(s, s', t_o). \]
This sets up a bijective correspondence between families 
of semigroup laws on $S$ parameterized by $T$, and structures
of $T$-semigroup scheme on $S \times T$.

For example, every algebraic semigroup law
$S \times S \to S$, $(s,s') \mapsto ss'$ defines a family of
semigroup laws on $S \times S$ parameterized by $S$, via
\[ \mu : S \times S \times S \longrightarrow S, \quad
(s,s',t) \longmapsto s t s'. \]
If $S$ is irreducible and complete, and $e \in S$ is idempotent,
then $\mu_e(s,s') = s s'$ in view of Lemma \ref{lem:idem}.
More generally, for any $t \in S$, we have
$\mu_t(s,s') = t \cdot ss'$,
with the notation of  Remark \ref{rem:cc} (iv). In other words,
the family $\mu$ arises from the action of the abelian variety
$G = eSe$ on $S$, defined in that remark. 

We now generalize this construction to obtain all families of 
semigroup structures on a complete irreducible variety, 
under a mild assumption on the parameter scheme.

\begin{theorem}\label{thm:fam}
Let $S$ be a complete irreducible variety, $T$ a connected 
scheme of finite type over $k$, and $\mu : S \times S \times T \to S$ 
a family of semigroup laws. Choose a $k$-point $t_o \in T$ 
and denote by $\kernel(S)$  the kernel of $(S,\mu_{t_o})$, 
by $\rho : S \to \kernel(S)$ the associated retraction, 
and by $G$ the associated abelian variety; recall that 
$G$ acts on $\kernel(S)$ by translations.

Then there exist unique morphisms 
$\varphi : \kernel(S) \times T \to S$ and $\gamma: T \to G$
such that 
\[ \mu(s,s',t) = \varphi(\mu_{t_o}(s,s'),t) \] 
on $S \times S' \times T$, and that the composition 
$\rho \circ \varphi :  \kernel(S) \times T \to \kernel(S)$
is the translation $(s,t) \mapsto \gamma(t) \cdot s$.

Conversely, given $\varphi: \kernel(S) \times T \to S$ 
such that there exists $\gamma : T \to G$ satisfying the 
preceding condition, the assignment 
$(s,s',t) \mapsto \varphi( \mu_{t_o}(s,s'),t)$ 
yields an algebraic semigroup law over $T$. Moreover,
$\varphi(s,t_o) = s$ on $\kernel(S)$, and 
$\gamma(t_o) = 1_G$. 
\end{theorem}

\begin{proof}
\smartqed
Denote for simplicity $\mu_{t_o}(s,s')$ by $ss'$.
We begin by showing that there exists a unique morphism
$\varphi : \kernel(S) \times T \to S$ such that
$\mu(s,s',t) = \varphi(ss',t)$. For this, we apply Lemma
\ref{lem:rig} and the subsequent Remark \ref{rem:rig} to 
the morphisms 
\[ \mu_{t_o} \times \id: 
S \times S \times T \to \kernel(S) \times T, \quad
\mu \times \id : S \times S \times T \to S \times T. \]
To check the  corresponding assumptions, note first that 
$\mu_{t_o}$ has a section 
\[ \sigma : \kernel(S) \longrightarrow S \times S, 
\quad s \longmapsto (s,s (e s e)^{-2} s),\]
where $e$ denotes a fixed idempotent of $(S,\mu_{t_o})$. 
(Indeed, let $\iota : X \times G \times Y \to S$ be the associated
closed immersion with image $\kernel(S)$. Then 
\[ \sigma (\iota (x,g,y)) = 
(\iota (x,g,y), \iota (x,1_G,y)) \]
as an easy consequence of Theorem \ref{thm:cisg}. Thus,
$\mu_{t_o} (\sigma ( \iota (x,g,y))) = \iota (x,g,y)$.)
By Lemma \ref{lem:retract}, it follows that the map
$\mu_{t_o}^{\#} : \cO_{\kernel(S)} \to (\mu_{t_o})_* (\cO_{S \times S})$ 
is an isomorphism.
Thus, $\mu_{t_o}$ satisfies the assumption (i)' of Remark \ref{rem:rig}; 
hence so does $\mu_{t_o} \times \id$. Also, the assumption (ii) of
Lemma \ref{lem:rig} holds for any point $(s,t_o)$ with 
$s \in \kernel(S)$, and the assumption (iii) of that lemma
holds as well, since $\sigma \times \id$ is a section of 
$\mu_{t_o} \times \id$. Finally, $S \times S \times T$ is connected, 
i.e., the assumption (iv)' of Remark \ref{rem:rig} is satisfied. 
Hence that remark yields the desired morphism $\varphi$.

In particular, $ss' = \mu(s,s',t_o) = \varphi(ss',t_o)$
for all $s,s' \in S$. Since the image of $\mu_{t_o}$ equals
$\kernel(S)$, it follows that $\varphi(s,t_o) = s$
for all $s \in \kernel(S)$. 

Next, consider the morphism
\[ \Psi := (\rho \circ \varphi) \times \id: 
\kernel(S,\mu_{t_o}) \times T \longrightarrow 
\kernel(S,\mu_{t_o}) \times T. \]
Then $\Psi_{t_o}$ is the identity by the preceding step; 
thus, $\Psi$ is an automorphism in view of Lemma
\ref{lem:end}. In other words, $\Psi$ arises from a morphism 
\[ \pi : T \longrightarrow \Aut(\kernel(S)), \quad 
t_o \longmapsto \id. \]
Since $T$ is connected, the image of $\pi$ is contained in
$\Aut^o(\kernel(S))$. We identify $\kernel(S)$ with
$X \times G \times Y$ via $\iota$. Then the natural map
\[ \Aut^o(X) \times \Aut^o(G) \times \Aut^o(Y) \longrightarrow
\Aut^o(\kernel(S)) \] 
is an isomorphism by \cite[Cor.~4.2.7]{BSU}. Moreover, 
$\Aut^o(G) \cong G$ via the action of $G$ on itself by translations,
see e.g. [loc.~cit.,Prop.~4.3.2]. Thus, we have
\[ \Psi(x,g,y,t) = (\alpha(x,t), g + \gamma(t), \beta(y,t),t)
 \]
for uniquely determined morphisms $\alpha : X \times T \to X$,
$\beta : Y \times T \to Y$ and $\gamma : T \to G$ such that
$\alpha \times \id$ is an automorphism of $X \times T$ over $T$,
and likewise for $\beta \times \id$.

We now use the assumption that $\mu$ is associative. This is 
equivalent to the condition that
\[ \varphi(s \varphi(s's'',t), t) = \varphi(\varphi(ss',t)s'',t) \]
on $S \times S \times S \times T$. Let $\psi := \rho \circ \varphi$,
then
\[ \psi(s \psi(s's'',t), t) = \psi(\psi(ss',t) s'',t) \]
on $\kernel(S) \times \kernel(S) \times \kernel(S) \times T$, 
since $ss' = \rho(s) \rho(s') = \rho(ss')$ on $S \times S$.
In view of the equalities 
$\psi(x,g,y) = (\alpha(x,t), g + \gamma(t), \beta(y,t))$
and $(x,g,y)(x',g',y') = (x,gg',y')$, the above associativity condition
for $\psi$ yields that $\alpha(x,t) = \alpha(\alpha(x,t),t)$
on $X \times T$, and $\beta(y,t) = \beta(\beta(y,t),t)$
on $Y \times T$. As $\alpha \times \id$ and $\beta \times \id$
are automorphisms, it follows that $\alpha(x,t) = x$ and
$\beta(y,t) = y$. Thus, 
\[ \psi(x,g,y,t) = (x, g + \gamma(t), y),\]
that is, $\rho \circ \varphi$ is the translation by $\gamma$.

For the converse, let $\varphi$, $\gamma$ be as in the statement. 
Then the morphism 
\[ \mu : S \times S \times T \longrightarrow S \times T, \quad
(s,s',t) \longmapsto \varphi(ss',t) \]
satisfies the associativity condition, since 
\[ \varphi(s \varphi(s's'',t), t) =
\varphi(\rho(s) \rho(\varphi(s's'',t)),t) =
\varphi(\gamma(t) \cdot \rho(s) \rho(s') \rho(s''), t) \]
and the right-hand side is clearly associative. Moreover,
as already checked, $\varphi(s,t_o) = s$ on $\kernel(S)$;
it follows that 
\[ \gamma(t_o) \cdot s = (\rho \circ \varphi)(s,t_o)
= \rho(s) = s \]
for all $s \in \kernel(S)$. Thus, $\gamma(t_o) = 1_G$.
\qed
\end{proof}

\begin{remark}\label{rem:fam}
(i) With the notation and assumptions of Theorem \ref{thm:fam},
one can easily obtain further results on the semigroup scheme structure 
of $S \times T$ over $T$ that corresponds to $\mu$, along the lines
of Theorem \ref{thm:cisg} and of Remark \ref{rem:cc}. For example,
one may check that the idempotent sections of the projection
$S \times T \to T$ are exactly the morphisms 
\[ T \longrightarrow S, \quad 
t \longmapsto \varphi(\gamma(t)^{-2} \cdot \varepsilon(t)), \] 
where $\varepsilon : T \to E(S,\mu_{t_o})$ is a morphism. In particular,
any such semigroup scheme has an idempotent section.

(ii) Consider the functor of composition laws on a variety $S$,
i.e., the contravariant functor from schemes to sets given by 
$T \mapsto \Hom(S \times S \times T,S)$; then the families of
algebraic semigroup laws yield a closed subfunctor
(defined by the associativity condition). When $S$ is projective, 
the former functor is represented by a quasiprojective 
$k$-scheme, 
\[ \CL(S) := \Hom(S \times S,S); \]
moreover, each connected component of $\CL(S)$ is of finite type 
over $k$ (as mentioned in Remark \ref{rem:mod}). 
Thus, the latter subfunctor is represented by a closed subscheme, 
\[ \SL(S) \subseteq \CL(S). \]
In particular, $\SL(S)$ is quasi-projective, and its connnected 
components are of finite type.

By Theorem \ref{thm:fam}, the connected component of 
$\mu_{t_o}$ in $\SL(S)$ is identified with the closed subscheme 
of $\Hom(\kernel(S),S) \times G$ consisting of those pairs
$(\varphi,\gamma)$ such that $\rho \circ \varphi$ is the 
translation by $\gamma$. Via the assignment
$(\varphi,\gamma) \mapsto (\gamma^{-1}\cdot \varphi,\gamma)$
(where $\gamma^{-1}\cdot \varphi$ is defined as in Remark 
\ref{rem:cc} (iv)), the above component of $\SL(S)$ 
is identified with the closed  subscheme of $\Hom(\kernel(S),S) \times G$ 
consisting of those pairs $(\sigma,\gamma)$ such that 
$\rho \circ \sigma = \id$,  that is, $\sigma$ is a section of $\rho$. 
This identifies the universal semigroup law on the above component, 
with the morphism 
\[ (s,s') \mapsto \gamma\cdot \sigma(\mu_{t_o}(s,s') ). \]
Note that the scheme of sections of $\rho$ is isomorphic to
an open subscheme of the Hilbert scheme, $\Hilb(S)$, by
assigning to every section its image (see \cite[p.~21]{Grothendieck}).  
This open subscheme is generally nonreduced, as shown by a
classical example where $S$ is a ruled surface over an elliptic
curve $C$. Specifically, $S$ is  obtained as the projective completion of 
a nontrivial principal $\bG_a$-bundle over $C$, 
and $\rho : S \to C$ is the ruling; then the section at infinity of $\rho$
yields a fat point of $\Hilb(S)$, 
as follows from obstruction theory (see e.g. \cite[Sec.~I.2]{Kollar}).
As a consequence, the scheme $\SL(S)$ is generally nonreduced as well.

(iii) The families of semigroup laws on further classes of varieties are worth
investigating. Following the approach of deformation theory, one may
consider those families of semigroup laws $\mu$ on a prescribed variety
$S$ that are parameterized by the spectrum of a local artinian $k$-algebra 
$R$ with residue field $k$, and that have a prescribed law $\mu_{t_o}$ 
at the closed point. Then the first-order deformations (i.e., those 
parameterized by $\Spec(k[t]/(t^2)$) form a $k$-vector space 
which may well be infinite-dimensional; this already happens when 
$S$ is the affine line, and $\mu_{t_o}$ the multiplication.
\end{remark}

\begin{acknowledgement} 
This article is an expanded and reorganized version of notes 
for lectures at ``The Workshop on Algebraic Monoids, 
Group Embeddings and Algebraic Combinatorics'' 
(Fields Institute, July 3-6, 2012). 
I thank the organizers of this workshop for giving me the 
opportunity to present this material, and all participants 
for fruitful contacts. Special thanks are due to Zhenheng Li 
for his careful reading of the article and his valuable comments.
\end{acknowledgement}

\end{document}